\documentclass[11pt,reqno]{article}
\usepackage[margin=0.96in, top=2.2cm,bottom=2.2cm]{geometry}

\usepackage[english]{babel}
\usepackage[T1]{fontenc}
\usepackage{setspace}
\usepackage{enumitem}
\usepackage[normalem]{ulem}
\usepackage{float}
\usepackage{scrextend}
\deffootnote{0em}{1.6em}{\thefootnotemark.\enskip}
\usepackage{xcolor}
\usepackage{latexsym}
\usepackage{csquotes}

\usepackage[makeroom]{cancel}
\usepackage{amssymb}
\usepackage[centertags]{amsmath}
\usepackage{amsthm}
\usepackage{mathtools}
\usepackage{euscript,mathrsfs}
\usepackage{bbm}
\usepackage[colorinlistoftodos]{todonotes}
\PassOptionsToPackage{hyphens}{url}
\usepackage[colorlinks=true, allcolors=blue]{hyperref}
\usepackage{titlesec}
\usepackage[title]{appendix}
\titlelabel{\thetitle.\quad}

\mathtoolsset{showonlyrefs}

\hypersetup{
    colorlinks,
    linkcolor={blue},
    citecolor={blue},
    urlcolor={blue}
}
\usepackage{yfonts}
\usepackage{indentfirst}
\usepackage[parfill]{parskip}
\usepackage[
    citestyle=numeric,
    bibstyle=authoryear,
    dashed=false,
    maxnames=4]{biblatex}
\makeatletter
\input{numeric.bbx}
\makeatother
\renewbibmacro{in:}{}
\DeclareNameAlias{author}{family-given}

\DeclareFieldFormat*{title}{#1}
\DeclareFieldFormat*[book]{title}{\textit{#1}}
\DeclareFieldFormat*{volume}{\textbf{#1}}

\newtheorem{Th}{Theorem}[section]

\newtheorem{Lem}[Th]{Lemma}
\newtheorem{Prop}[Th]{Proposition}
\newtheorem{Cor}[Th]{Corollary}
\theoremstyle{definition}
\newtheorem{remark}[Th]{Remark}

\numberwithin{equation}{section}

\theoremstyle{definition}

\DeclareMathOperator*{\argmin}{arg\,min}
\newcommand{\fc}{\mathcal{F}}

\newcommand{\ca}{\mathcal{A}}
\newcommand{\cu}{\mathcal{U}}

\newcommand{\pr}{\mathbb{P}}
\newcommand{\ex}{\mathbb{E}}
\newcommand{\eps}{\varepsilon}

\title{\textsc{A Sequential Testing Problem with Signal Control}}
\author{
\textsc{Steven Campbell} 
\thanks{ 
\,\textsc{Columbia University, Department of Statistics, 1255 Amsterdam Ave, New York, NY 10027, USA} (e-mail: {\it sc5314@columbia.edu})}
\and
\textsc{Georgy Gaitsgori} 
\thanks{ 
\,\textsc{Columbia University, Department of Mathematics, 2990 Broadway, New York, NY 10027, USA} (e-mail: {\it gg2793@columbia.edu})}
\and 
\textsc{Richard Groenewald} 
\thanks{ 
\,\textsc{Columbia University, Department of Statistics, 1255 Amsterdam Ave, New York, NY 10027, USA} (e-mail: {\it rag2202@columbia.edu})}
}
\date{\today}

\begin{document}
\maketitle

\begin{abstract}
We study a controlled version of the Bayesian sequential testing problem for the drift of a Wiener process, in which the observer exercises discretion over the signal intensity. This control incurs a running cost that reflects the resource demands of information acquisition. The objective is to minimize the total expected cost, combining both the expenditure on control and the loss from misclassifying the unknown drift. By allowing for a general class of loss functions and \emph{any measurable} cost of control, our analysis captures a broad range of sequential inference problems.
We show that when a function, determined by the cost structure, admits a global minimizer, the optimal control is constant and explicitly computable, thereby reducing our setting to a solvable optimal stopping problem. If no such minimizer exists, an optimal control does not exist either, yet the value function remains explicit. Our results thus demonstrate that full tractability can be retained even when extending sequential inference to include endogenous control over the information flow.
\end{abstract}



\section{Introduction}

Consider a scenario in which a quantity of interest, modeled by a binary random variable $\theta$, cannot be observed directly. Instead, the observer receives a continuous stream of noisy, modulated data
\begin{equation}
    X(t) = \theta \int_0^t u(s) \, ds + W(t), \quad 0 \le t < \infty,
\end{equation}
where $W(\cdot)$ is a standard Brownian motion independent of $\theta$, and $u(\cdot)$ is a \textit{signal intensity} selected by the observer. 
The objective is to select an appropriate control $u(\cdot)$ and choose a stopping time $\tau$ at which to infer the true value of $\theta$ from the information in the signal process $X(\cdot)$, as accurately and cost-effectively as possible.

The trade-off between estimation accuracy and the cost of information is captured through a loss function $\mathfrak{L}(\cdot,\cdot)$ and two running cost components: a constant temporal cost $c$, and a control-dependent cost $\varphi(u(\cdot))$, both incurred per unit of time. The observer seeks to minimize the expected total cost
\begin{equation}
    \ex \left[\mathfrak{L}(\theta, d) + \int_0^\tau \varphi(u(s)) \, ds + c\tau \right],
\end{equation}
where the decision $d$ about $\theta$ is made at time $\tau$ based on the path $(X(t))_{0\leq t\leq \tau}$. In this way, the observer faces a \textit{triple problem of sequential estimation (filtering), optimal control, and stopping}.

Interestingly, this problem admits an explicit and structurally simple solution under broad assumptions on the cost functions. Specifically, if a certain function $\eta(\cdot)$ (defined in \eqref{definition_eta_function}) attains its infimum over the state space of admissible controls, then the optimal control is constant and can be computed explicitly. Moreover, the experiment should be terminated when the optimally controlled process $X(\cdot)$ reaches one of two computable thresholds. If $\eta(\cdot)$ fails to attain a minimum, then no optimal control exists. However, in both cases, the value function can still be computed explicitly. Notably, the resulting value function closely resembles that of the classical sequential testing problem (see \cite{CampZha24, Shiryaev67}), and the structure of the solution persists irrespectively of the choice of $\varphi(\cdot)$, which may be \textit{any} measurable function. This generality highlights a surprising robustness: even for arbitrary measurable control costs, the structure of the optimal policy remains tractable and interpretable.

Our approach applies standard filtering techniques to reduce the original problem to one of stochastic control with discretionary stopping, and handles the latter using variational inequalities. It is important to highlight that although the original problem is formulated in the ``strong sense,'' we analyze and solve our triple problem in its ``weak formulation,'' which is more amenable to analysis. However, since the optimal control (when it exists) turns out to be constant, the strong and weak formulations ultimately yield the same solution. We complement this result with an explicit characterization of the statistical properties of the optimal strategy, enabling quantitative assessment of its performance and sensitivity to model inputs.

Despite its abstract formulation, our problem is closely aligned with real-world decision-making under uncertainty. Many modern inference tasks demand real-time decisions based on partial and noisy information. In domains such as adaptive experimentation, systems monitoring, and diagnostic testing, decision-makers must balance the cost of acquiring additional data against the urgency and accuracy of their conclusions. For example, in disease screening or cognitive testing, clinicians aim to determine a binary outcome (e.g., the presence or absence of a condition) from a stream of evidence. Tests may vary in both resolution and cost, and decisions frequently carry time-sensitive consequences. These settings naturally give rise to problems of sequential inference under resource constraints.

The remainder of the paper is organized as follows. The next subsection reviews related work. Section \ref{sec_model} develops the formal setup and carries out the key reduction---via filtering---from a full inference problem to a joint control-and-stopping formulation in weak form. In Section \ref{sec_prem_analysis}, we begin the analysis, introduce the relevant variational inequalities, and propose a candidate solution structure. Section \ref{sec_main_results} contains our main results and their proofs. Finally, Section \ref{sec_examples} presents several illustrative examples and discussion.

\subsection{Related work}

As noted above, our problem combines all three features of sequential estimation (filtering), optimal stopping, and control, making it a rare instance of an explicitly solvable \textit{triple} problem. To the best of our knowledge, only a handful of other works address such a combination \cite{BayKra15,DalShi,DaySez16,EksKar,EksLinOlo,MilEks24,HarSun}. The first is due to Dalang and Shiryaev \cite{DalShi}, who study a quickest detection problem with signal control. More recently, Milazzo and Ekström \cite{MilEks24} study a comparable model with irreversible control over the observation rate. The work of Dayanik and Sezer \cite{DaySez16} is also pertinent, though of a somewhat different nature: they examine a setting with discrete controls where sensors can be installed sequentially to improve observation quality. Whereas, Bayraktar and Kravitz \cite{BayKra15} study detection with a costly and finite number of available observations that must be allocated over time, and De Angelis et al.\ \cite{de2022quickest} allow for costly inspections that may produce false negatives. For a two player setting, we point to the work of Ekström, Lindensjö, and Olofsson \cite{EksLinOlo}, which analyzes a non-zero-sum stochastic game involving both control and stopping.

Two further examples are especially close in spirit to our study. Harrison and Sunar \cite{HarSun} consider a similar structural framework but impose a different cost functional with exponential discounting and restrict the control process to take finitely many values. Ekström and Karatzas \cite{EksKar}, by contrast, study a sequential estimation problem akin to ours, allowing for a general prior on the latent variable $\theta$ and employing an $L^2$ loss to measure estimation accuracy. Surprisingly, despite substantial differences in both our assumptions and solution methods, the structure of the optimal strategies in our paper mirrors that of \cite{EksKar}. Collectively, these two works suggest that explicit solutions to triple problems in sequential inference are possible when one balances generality in the prior with structure in the cost, or vice versa, generality in the cost and more restrictive assumptions on the prior.

When control over the signal is disallowed, the problem reduces to the classical setting of sequential testing for the drift of an arithmetic Brownian motion. This was first solved in continuous time by Shiryaev \cite{Shiryaev67} and later extended to more general classification losses by Campbell and Zhang \cite{CampZha24}. Variants of this problem---featuring different assumptions on the cost structure, observation dynamics, or time horizon---have been extensively studied. Without aiming for completeness, we refer to \cite{CGGK25, EksKarVai22, EksVai15, EksWang24, ErnPes24, ErnPesZho20, GapPes04, GapShi11, JohPes18, Johnson22, PesShi00, ZhiShi11}, and the references therein.

For problems that combine control and stopping but not filtering, we refer the reader to the seminal work of Dubins and Savage \cite{DubSav}, as well as the foundational monographs of Krylov \cite{Krylov80}, El Karoui \cite{ElKaroui81}, and Bensoussan and Lions \cite{BenLio82}. In the spirit of our results, we also highlight a number of works that provide explicit solutions to control-and-stopping problems: see \cite{DavZer, KarOcoWanZer} for examples involving singular control, and \cite{BGK25, KarOco02, KarSud99, KarSud01, OcoWeer08} for settings involving drift and/or volatility control.

\section{Model}\label{sec_model}

We are now ready to formalize the problem. As already mentioned, we begin with a ``strong'' formulation, which is the natural (and necessary) setting for the filtering arguments that follow. These yield an equivalent Markovian representation that can then be recast in a ``weak'' sense, which is more amenable to analysis.

\subsection{Strong formulation}\label{subsec_strong_formlation}

Let $(\Omega, \mathcal{F}, \mathbb{F}=(\mathcal{F}_t)_{t\geq0},\mathbb{P})$ be a filtered probability space satisfying the usual conditions and supporting both a Bernoulli random variable $\theta$ with $\mathbb{P}(\theta = 1) = p = 1 - \mathbb{P}(\theta = 0)$ for some $0 < p < 1$, and an independent $(\mathbb{P},\mathbb{F})$-Brownian motion $W(\cdot) = (W(t))_{t \ge 0}$. On this space, we define a controlled process $X^u(\cdot) = (X^u(t))_{t \ge 0}$ satisfying 
\begin{equation}\label{controlled_sde}
    X^u(t) = \theta \int_0^t u(s) \, ds + W(t), \quad 0 \le t < \infty,
\end{equation}
where $u(\cdot) = (u(t))_{t \ge 0}$ is an $\mathbb{F}^{X^u}$--progressively measurable control process, taking values in a non-empty measurable subset $\mathcal{U} \subseteq \mathbb{R}\setminus\{0\}$, and satisfying the square integrability condition 
\begin{equation}\label{eqn:sq.int.u}
    \mathbb{E}\left[\int_0^tu^2(s) \, ds\right]<\infty, \quad \forall t\geq0.
\end{equation}
By $\mathbb{F}^{X^u} \coloneqq (\mathcal{F}^{X^u}(t))_{t \ge 0}$ we mean the usual augmentation of the filtration generated by $X^u(\cdot)$. Evidently, $u(\cdot)$ represents the signal intensity that governs the information revealed by $X^u(\cdot)$ about $\theta$. Since $u=0$ amounts to halting the test, we disallow this choice. As we will discuss in a Remark \ref{rem_zero_control} below, located after the proof of Theorem \ref{thm_main}, 
this decision is made mainly for convenience of exposition.
We denote by $\mathcal{T}^{X^u}$ the collection of all $\mathbb{F}^{X^u}$--stopping times.

\begin{remark}
    We allow for general measurable subsets $\mathcal{U} \subseteq \mathbb{R}\setminus\{0\}$. Canonical choices include $\mathcal{U} = (0,\infty)$ for positive controls, or $\mathcal{U} = (a, b]$ with $0 \le a < b < \infty$ for bounded controls. Our analysis is agnostic as to the specific choice of $\mathcal{U}$.
\end{remark}

\begin{remark}\label{rem:not.circular.strong.formulation}
    At first glance, coupling \eqref{controlled_sde} with the stipulation that $u(\cdot)$ be $\mathbb{F}^{X^u}$--progressively measurable may look circular.  However, this requirement is benign and can be understood as follows. Formally, an admissible control $u(\cdot)$ is an $\mathbb F$--progressively measurable, $\mathcal U$--valued process satisfying \eqref{eqn:sq.int.u} for which \eqref{controlled_sde} has a strong solution $X^u(\cdot)$ and with respect to which $u(\cdot)$ is $\mathbb F^{X^u}$--progressively measurable. 
    This admissible class is not only non-empty (it contains all constant controls), but also sufficiently rich: if we place ourself on a canonical space, this class contains the subclass of all controls that can be represented in terms of a jointly measurable non-anticipative path functional $F: \mathbb{R}_+ \times C([0, \infty)) \to \mathcal{U}$ that is Lipschitz and has linear growth on $C[0,T]$ for every finite $T>0$; in this case, $u(t) = F(t,X^u(\cdot\wedge t)), 0 \le t < \infty$ is admissible.
    Indeed, under these assumptions \eqref{controlled_sde} has a unique strong solution and $u(\cdot)$ is $\mathbb{F}^{X^u}$--progressively measurable (cf. \cite[Theorem 1]{cont2020support} and \cite[Definition 2.1]{cont2013functional}). We do not pursue these technicalities further, as they disappear in the weak formulation.
\end{remark}

\begin{remark}
    The main purpose of the square-integrability condition \eqref{eqn:sq.int.u} is to ensure the validity of the filtering lemma (Lemma \ref{lem_filtering}) below. Once we pass to the weak formulation (Section~\ref{subsec_weak_formulation}), this assumption can be relaxed without without affecting the rest of the analysis. However, it is needed to preserve the link between the weak formulation and the filtering interpretation in the strong setting.
\end{remark}

Let $\mathcal{A}$ denote the collection of all admissible processes $X^u(\cdot)$ arising from such controls. 
Our goal is to find an optimal triple $(X^{u^*}, \tau^*, d^*)$, where $X^{u^*} \in \mathcal{A}$, $\tau^* \in \mathcal{T}^{X^{u^*}}$, and $d^*$ is an $\mathcal{F}^{X^{u^*}}(\tau^*)$--measurable random variable $d^*: \Omega \to \mathcal{D}\subseteq [0,1]$, that minimizes the expected cost
\begin{equation}\label{def_expected_loss_X}
    J\big(X^u(\cdot), \tau, d\big) 
    \coloneqq
    \mathbb{E} \left[\mathfrak{L}(\theta, d) + \int_0^\tau \varphi(u(s)) \, ds + c\tau \right],
\end{equation}
over all admissible triples $(X^u, \tau, d)$ with $X^u \in \mathcal{A}$, $\tau \in \mathcal{T}^{X^u}$, and $d$ measurable with respect to $\mathcal{F}^{X^u}(\tau)$ and valued in some compact $\mathcal{D} \subseteq [0,1]$.

Here, $c > 0$ is an ``operating cost'' per unit of time, $\varphi: \mathcal{U} \to \mathbb{R}$ is a running cost of control, and $\mathfrak{L}: \{0, 1\} \times [0,1] \to \mathbb{R}_+$ is a loss function, satisfying $\mathfrak{L}(x, y) = 0$ if and only if $x = y$. The set $\mathcal{D}$ represents the allowable values for the final decision $d$ about $\theta$; it may be discrete (e.g., $\mathcal{D} = \{0,1\}$, corresponding to the ``hard'' classification) or permit a continuous range of decisions (e.g., $\mathcal{D} = [0,1]$, corresponding to the ``soft'' classification). 
At this stage, we only ask $\varphi(\cdot)$ to be a measurable function and $d\mapsto \mathfrak{L}(\theta,d)$ to be continuous on $\mathcal{D}$ for each $\theta\in\{0,1\}$. If $\mathcal{D}$ is a discrete subset of $[0,1]$, this is trivially true. Some additional assumptions on $\mathfrak{L}(\cdot,\cdot)$ will be specified later in this section.

To solve the minimization problem \eqref{def_expected_loss_X}, we perform three reductions. First, we reformulate the problem in terms of the posterior probability process. Second, we apply iterated conditioning to eliminate the decision rule $d$ from the objective. Finally, we pass to a ``weak analogue'' of the resulting formulation and embed the problem in a general Markovian setting that accommodates arbitrary initial beliefs.

\subsection{Reformulation in terms of posterior probability}

As is common in sequential testing problems, we introduce the posterior probability process $\Pi^u(\cdot) = (\Pi^u(t))_{t \ge 0}$, defined by
\begin{equation}\label{def_cond_prob_proccess}
    \Pi^u(t) \coloneqq \mathbb{P}\left(\theta = 1 \mid \mathcal{F}^{X^u}(t)\right), \quad 0 \le t < \infty.
\end{equation}
The next lemma summarizes its essential features.

\begin{Lem}\label{lem_filtering}
    The process $\Pi^u(\cdot)$ defined in \eqref{def_cond_prob_proccess} satisfies
    \begin{equation}\label{cond_prob_dynamics}
        \Pi^u(t) = p + \int_0^t u(s)\, \Pi^u(s)(1 - \Pi^u(s))\, dB^u(s), \quad 0 \le t < \infty,
    \end{equation}
    where
    \begin{equation}\label{innovation_process}
        B^u(t) \coloneqq X^u(t) - \int_0^t u(s)\, \Pi^u(s)\, ds, \quad 0 \le t < \infty,
    \end{equation}
    is a $(\mathbb{P}, \mathbb{F}^{X^u})$--Brownian motion known as the innovation process.
\end{Lem}

\begin{proof}
    The statements of Lemma \ref{lem_filtering} are well-known results in filtering theory (see e.g., \cite[Theorem 8.1]{LS01}). For detailed proofs, we refer the reader to Chapter 8 of Lipster and Shiryaev \cite{LS01}, Chapter 3 of Bain and Crisan \cite{BainCrisan}, 
    or Chapter 6 of the lecture notes by Chigansky \cite{Chigansky}.
\end{proof}

\begin{remark}
    The filtering theory underlying Lemma~\ref{lem_filtering} hinges on the control $u(\cdot)$ being $\fc^{X^u}$--measurable. This is precisely why we begin with the strong formulation.
\end{remark}

\begin{remark}\label{rem:gen.drift.vol.filtering.eqn}
If, instead of \eqref{controlled_sde}, the observation process $X^u(\cdot)$ satisfies 
\begin{equation}
    X^u(t) = \mu\,\theta \int_0^t u(s)\,ds + \sigma W(t),\quad 0 \le t < \infty \qquad \text{for some $\mu \neq 0, \ \sigma > 0$},
\end{equation}
then the corresponding innovation process 
$
    B^u(t) \coloneqq \frac{1}{\sigma} \left(X^u(t) - \mu \int_0^t u(s)\, \Pi^u(s)\, ds \right), \ 0 \le t < \infty
$
is again a $(\mathbb{P},\mathbb{F}^{X^u})$--Brownian motion; and the filtering equation of Lemma \ref{lem_filtering} reads
\begin{equation}\label{cond_prob_dynamics_scaled}
    \Pi^u(t) = p + \frac{\mu}{\sigma} \int_0^t u(s)\, \Pi^u(s)(1 - \Pi^u(s))\, dB^u(s), \quad 0 \le t < \infty.
\end{equation}
Thus, the only effect of changing the magnitude of the signal and noise parameters $(\mu,\sigma)$ is the multiplicative factor $\mu/\sigma$ in the diffusion coefficient of $\Pi^u(\cdot)$.
\end{remark}

The process $\Pi^u(\cdot)$ has the very nice, so-called ``identifiability,'' property that it is a consistent estimator of $\theta$ if enough signal is generated. We make note of this here and provide the proof for completeness in Appendix \ref{subsec_proof_of_consistency_pi}.

\begin{Lem}\label{lem:pi.consistent}
    Suppose $\lim_{t\to\infty}\int_0^tu(s)^2 \, ds \to\infty$, $\mathbb{P}$-a.s. Then, $\lim_{t\to\infty}\Pi^u(t)=\theta$, $\mathbb{P}$-a.s.
\end{Lem}

We can now recast the problem of minimizing \eqref{def_expected_loss_X} in terms of the posterior process $\Pi^u(\cdot)$. Specifically, we apply the tower property of conditional expectations to rewrite the cost as
\begin{equation}
    \begin{split}
        \mathbb{E} \bigg[
            \mathfrak{L}(\theta, d)&+\int_0^\tau \varphi(u(s)) \, ds + c\tau
        \bigg]
        =
        \mathbb{E}\left[
            \mathbb{E} \left[
                \mathfrak{L}(\theta, d)+\int_0^\tau \varphi(u(s)) \, ds + c\tau
            \Big| \, \mathcal{F}^{X^u}(\tau)
        \right]\right]
        \\&=
        \mathbb{E}\left[
            \mathbb{E} \left[
                \mathfrak{L}(1, d) \mathbf{1}_{\{\theta = 1\}} +
                \mathfrak{L}(0, d) \mathbf{1}_{\{\theta = 0\}}
                \,\Big|\, \mathcal{F}^{X^u}(\tau)
            \right]
            +
            \int_0^\tau \varphi(u(s)) \, ds + c\tau
        \right]
        \\&=
        \mathbb{E}\left[
            \mathfrak{L}(1, d)\, \Pi^u(\tau) + \mathfrak{L}(0, d)\, (1 - \Pi^u(\tau))
            +
            \int_0^\tau \varphi(u(s)) \, ds + c\tau
        \right].
    \end{split}
\end{equation}

Next, we observe that once the stopping time $\tau$ is fixed, the optimal decision rule $d^*$ should minimize the function $d \mapsto \mathfrak{L}(1, d)\, \pi + \mathfrak{L}(0, d)\, (1 - \pi)$, evaluated at $\pi = \Pi^u(\tau)$. Since $\mathfrak{L}(1,\cdot)$ and $\mathfrak{L}(0,\cdot)$ are continuous on $\mathcal{D}$, the function
\begin{equation}
    f(\pi,d) \coloneqq \mathfrak{L}(1,d)\, \pi + \mathfrak{L}(0,d)\, (1 - \pi)
\end{equation}
is a \textit{Carathéodory} function: Borel measurable in its first argument and continuous in its second. By the compactness of $\mathcal{D}$ and Theorem 18.19 of Aliprantis and Border \cite{guide2006infinite}, there exists a measurable selector $h(\pi) \in \argmin_{d \in \mathcal{D}} f(\pi,d)$. That is, we may write $d^* = h(\Pi^u(\tau))$ for some measurable function $h: (0,1) \to \mathcal{D}$. As a result, the expected cost simplifies to
\begin{equation}\label{objective_in_terms_of_pi}
    \mathbb{E}\left[
        g(\Pi^u(\tau)) + \int_0^\tau \varphi(u(s)) \, ds + c\tau
    \right],
\end{equation}
where $g: (0,1) \to \mathbb{R}_+$ is defined by $g(\pi) \coloneqq f(\pi, h(\pi))$. Importantly, since $\pi \mapsto f(\pi, d)$ is affine for each fixed $d$, the function $g(\cdot)$ is concave as the pointwise minimum of affine functions. In addition, by using the representation \eqref{objective_in_terms_of_pi} we can formally allow stopping at infinity by defining:
\begin{equation}g(\Pi^u(\infty)) + \int_0^\infty (\varphi(u(s))+c)ds=\liminf_{t\to\infty}\left[g(\Pi^u(t)) + \int_0^t (\varphi(u(s))+c)ds\right].\end{equation}

Although the new objective \eqref{objective_in_terms_of_pi} is written in terms of the posterior process $\Pi^u(\cdot)$, the optimization is carried out over control processes and stopping times adapted to the filtration generated by $X^u(\cdot)$. If one assumes that the control process $u(\cdot)$ is strictly positive (or strictly negative), then standard results (e.g., those used in the proof of Lemma~\ref{lem_filtering}) imply that $\Pi^u(\cdot)$ and $X^u(\cdot)$ generate the same filtration. This would justify reducing the problem to one of minimizing \eqref{objective_in_terms_of_pi} over controls and stopping times adapted to $\mathbb{F}^{\Pi^u}$. To maintain generality, we avoid imposing this assumption on admissible controls. This motivates a reformulation in the ``weak'' sense. As we show later, this transition is without loss of generality: the optimal values and policies arising from the weak and strong formulations will be equivalent.

\subsection{Weak Markovian formulation}\label{subsec_weak_formulation}

We now work on an \emph{arbitrary} filtered probability space $(\Omega^u, \mathcal{F}^u, \mathbb{F}^u = (\mathcal{F}^u(t))_{t \ge 0}, \mathbb{P}^u)$ satisfying the usual conditions that supports a $(\mathbb{P}^u, \mathbb{F}^u)$--Brownian motion $B^u(\cdot) = (B^u(t))_{t \ge 0}$. We denote by $\mathbb{E}^u[\, \cdot \, ]$ expectations with respect to the measure $\mathbb{P}^u$. While some notation introduced here is reused from the strong formulation above, we note that it refers to an entirely separate construction.

For each $\pi \in [0, 1]$, we define the controlled process $\Pi^u_\pi(\cdot) = (\Pi^u_\pi(t))_{t \ge 0}$ taking values in $[0, 1]$, and governed by the dynamics 
\begin{equation}\label{controlled_dynamicccs_pi}
    \Pi^u_\pi(t) = \pi + \int_0^t u(s)\, \Pi^u_\pi(s)\, (1-\Pi^u_\pi(s)) \, dB^u(s), \quad 0 \le t < \infty,
\end{equation}
where $u(\cdot) = (u(t))_{t \ge 0}$ is an $\mathbb{F}^u$--progressively measurable control process taking values in a measurable subset $\mathcal{U} \subseteq \mathbb{R}\setminus\{0\}$ such that the stochastic integral in \eqref{controlled_dynamicccs_pi} is well defined. As in Remark \ref{rem:not.circular.strong.formulation}, we emphasize that this construction is not circular and deliberately broad: it defines a large admissible class which includes, but is not limited to, progressively measurable $\mathcal{U}$-valued processes satisfying \eqref{eqn:sq.int.u}. For definiteness, one may impose \eqref{eqn:sq.int.u} as in Section \ref{subsec_strong_formlation}, although this condition plays no further role in the analysis.

We denote by $\mathcal{A}(\pi)$ the set of all “admissible” processes $\Pi^u_\pi(\cdot)$ arising as weak solutions of \eqref{controlled_dynamicccs_pi}. By this we mean that each element of $\mathcal{A}(\pi)$ is a quadruple
\begin{equation}
  \bigl((\Omega^u,\mathcal{F}^u,\,\mathbb{F}^u,\mathbb{P}^u),\,B^u(\cdot),\,u(\cdot),\,\Pi^u_\pi(\cdot)\bigr),
\end{equation}
where (i) $(\Omega^u,\mathcal{F}^u,\mathbb{F}^u, \mathbb{P}^u)$ is a filtered probability space supporting a $(\mathbb{P}^u,\mathbb{F}^u)$--adapted Brownian motion $B^u(\cdot)$, (ii) $u(\cdot)$ is an $\mathbb{F}^u$--progressively measurable, $\mathcal{U}$--valued process, and (iii) $\Pi^u_\pi(\cdot)$ is $\mathbb{F}^u$--adapted and satisfies the SDE above. Note that $\mathcal{A}(\pi)$ is nonempty, since for any constant control $u(\cdot)\equiv u\in\mathcal{U}$ one trivially obtains a weak solution. We will often abuse notation and write $\Pi^u_\pi(\cdot)\in\mathcal{A}(\pi)$,
with the understanding that this refers to the full quadruple above. Finally, for each $\Pi^u_\pi(\cdot)\in\mathcal{A}(\pi)$, we denote by $\mathcal{T}^u$ the collection of all stopping times with respect to $\mathbb{F}^u$.

With the above notation, we now consider the problem of finding a process $\Pi^{u^*}_\pi(\cdot) \in \mathcal{A}(\pi)$ and an associated stopping time $\tau^* \in \mathcal{T}^u$ such that the pair $(\Pi^{u^*}_\pi(\cdot), \tau^*)$ minimizes the expected cost
\begin{equation}\label{def_expected_loss_pi}
    J\big(\pi, \Pi^u_\pi(\cdot), \tau\big) 
    \coloneqq
    \mathbb{E}^u \left[
        g\big(\Pi^u_\pi(\tau)\big)
        + \int_0^\tau \varphi(u(s)) \, ds
        + c\tau
    \right],
\end{equation}
over all $\Pi^u_\pi(\cdot) \in \mathcal{A}(\pi)$ and $\tau \in \mathcal{T}^u$. Here, $c > 0$ is a constant ``operating cost'' per unit time, $\varphi: \mathcal{U} \to \mathbb{R}$ is the running cost of control, and $g: [0, 1] \to \mathbb{R}_+$ is the terminal penalty function. It is worth emphasizing once more that $\varphi(\cdot)$ may be \emph{any} measurable function on $\mathcal{U}$. 

We now define the second-order differential operator $\mathcal{L}$ by
\begin{equation}
    (\mathcal{L}f)(\pi) \coloneqq \pi^2 (1 - \pi)^2 f''(\pi), \quad f \in C^2((0, 1)),
\end{equation}
which coincides with the generator of the posterior process with constant control $ u(\cdot) \equiv 2 $. With this notation, we impose the following structural assumptions on $g(\cdot)$:

\begin{enumerate}[label = \textbf{(A\arabic*)}, resume]
    \item \label{ass_1}
    The function $g(\cdot)$ is symmetric about $1/2$, concave, and satisfies $g(0) = g(1) = 0$. Moreover, either:
    \begin{enumerate}
        \item $g(\pi) \propto \pi \wedge (1 - \pi)$, or
        \item $g(\cdot) \in C^2((0,1))$ and the function $(\mathcal{L}g)(\cdot)$ is strictly decreasing on $(0,1/2)$ and strictly increasing on $(1/2,1)$.
    \end{enumerate}
\end{enumerate}

Recall that our assumptions on the loss function $\mathfrak{L}(\cdot, \cdot)$ already imply that $g(\cdot)$ is concave and $g(0)=g(1)=0$. Assumption \ref{ass_1} introduces several additional conditions on $g(\cdot)$ to facilitate the analysis. These assumptions are satisfied by a broad class of commonly used loss functions and, in many cases, are made primarily for convenience of exposition rather than necessity.

In particular, while we assume $g(\cdot)$ is symmetric about $1/2$, all of our results remain valid for asymmetric alternatives (e.g., $g(\pi) = a\pi \wedge b(1 - \pi)$ for $a, b > 0$), with only minor (but tedious) modifications to the arguments. For the sake of clarity and to avoid technical bookkeeping, we focus on the symmetric case. 
Aside from sufficient regularity, the only essential additional assumption is the monotonicity of $(\mathcal{L}g)(\cdot)$, which ensures that the continuation region of the associated stopping problem is an interval---this structure plays a key role in our analysis.

We now list several examples of commonly used loss functions $\mathfrak{L}(\cdot,\cdot)$ whose induced penalties $g(\cdot)$ satisfy the above conditions.  The choice $g(\pi) = \pi \wedge (1 - \pi)$ arises from Shiryaev's classical sequential testing problem \cite{Shiryaev67}, corresponding to the standard Bayesian classification loss $\mathfrak{L}(\theta, d) = \mathbf{1}_{\{d \neq \theta\}}$. Other examples include:
\begin{itemize}
    \item The \textit{cross-entropy loss} $\mathfrak{L}(\theta, d) = -\theta \ln d - (1 - \theta) \ln(1 - d)$, which leads to $g(\pi) = -\pi \ln \pi - (1 - \pi) \ln(1 - \pi)$, with the convention $g(0) = g(1) = 0$;
    \item The \textit{$L^1$ loss} $\mathfrak{L}(\theta, d) = |d - \theta|$, which yields $g(\pi) = \pi \wedge (1 - \pi)$;
    \item The \textit{$L^2$ loss} $\mathfrak{L}(\theta, d) = (d - \theta)^2$, resulting in $g(\pi) = \pi(1 - \pi)$.
\end{itemize}

From now on, we study the problem of minimizing \eqref{def_expected_loss_pi}, and denote the value function of this problem by
\begin{equation}\label{def_value_function}
    V(\pi) \coloneqq \inf\limits_{\substack{\Pi^u_\pi(\cdot) \in \mathcal{A}(\pi) \\ \tau \in \mathcal{T}^u}} J(\pi, \Pi^u_\pi(\cdot), \tau).
\end{equation}

\section{Preliminary Analysis}\label{sec_prem_analysis}

In this section, we begin our analysis of the problem \eqref{def_value_function} by making several conjectures about the structure of its solution. We also introduce necessary notation and prove auxiliary results that will support the main arguments to follow.

We start by formulating a system of variational inequalities that the value function $V(\cdot)$ is expected to satisfy for all $\pi \in [0, 1]$ (wherever $V''(\cdot)$ exists):
\begin{enumerate}[label = \textbf{(\roman*)}]
    \item $g(\pi) - V(\pi) \ge 0$; \label{var_ineq_1}
    \item $\inf\limits_{u \in \mathcal{U}}\left[\frac{1}{2} V''(\pi)\big(u \pi(1-\pi)\big)^2 + \varphi(u) \right] + c \ge 0$; \label{var_ineq_2}
    \item $\left( g(\pi) - V(\pi) \right)\left( \inf\limits_{u \in \cu}\left[\frac{1}{2} V''(\pi)\big(u \pi(1-\pi)\big)^2 + \varphi(u) \right] + c \right) = 0$. \label{var_ineq_3}
\end{enumerate}
Inequality \ref{var_ineq_1} reflects the fact that the value function must be dominated by the loss function $g(\cdot)$, since one can always stop immediately (i.e., choose $\tau \equiv 0$) and pay $g(\pi)$. Inequality \ref{var_ineq_2} serves as a Bellman-type condition that characterizes the evolution of the value function when exerting control. Finally, the complementarity condition \ref{var_ineq_3} asserts that, at each point in time, either it is optimal to stop or the process should be controlled in a way that maintains equality in \ref{var_ineq_2}. This logic can also be used to derive a characterization of the optimal pair $(\Pi^{u^*}(\cdot), \tau^*)$.

A standard approach to solving variational inequalities, i.e., finding a candidate value function $\widehat{V}(\cdot)$ that satisfies \ref{var_ineq_1}--\ref{var_ineq_3}, is to first compute the pointwise minimizer $\widehat{u} = \widehat{u}(\pi)$ of the expression in \ref{var_ineq_2}. Substituting this back into \ref{var_ineq_2} yields an ordinary differential equation (ODE) for $\widehat{V}(\cdot)$. One typically solves this ODE on the continuation region $\mathcal{C}:=\{\pi \in (0,1) : g(\pi) > V(\pi)\}$ and then ``glues'' the resulting solution to $g(\cdot)$ so that the full function satisfies conditions \ref{var_ineq_1} and \ref{var_ineq_3}. However, in our case, this procedure leads to a  non-linear, non-standard differential equation, which is not amenable to analysis.

To circumvent these difficulties, we instead conjecture a specific structure for the value function in the continuation region:
\begin{equation}\label{conjectured_form_of_value_function}
    V''(\pi) = -\alpha \, (\pi (1 - \pi))^{-2}
\end{equation}
for some constant $\alpha > 0$. If this were true, then for $\pi\in\mathcal{C}$ (i.e., when the control is active) the variational inequality \ref{var_ineq_2} simplifies to
\begin{equation}\label{var_ineq_2_simplified}
    \inf_{u \in \mathcal{U}} \left[ -\frac{\alpha}{2} u^2 + \varphi(u) \right] + c \ge 0,
\end{equation}
which implies that the optimal control $u^*(\cdot)$ (when it exists) must be constant in both time and state. That is, $u^*(\cdot) \equiv u^*$ for all $t \ge 0$ and some $u^* \in \mathcal{U}$ that depends on the constant $\alpha$ of \eqref{conjectured_form_of_value_function}.

The simplified inequality \eqref{var_ineq_2_simplified} now provides a way to find the optimal control $u^*(\cdot)$, but only if the value of $\alpha$ in \eqref{conjectured_form_of_value_function} is known. However, under the assumption that the optimal control is constant, we can determine its value through an alternative route. Namely, for a fixed control $u(\cdot) \equiv u$, with $u \in \mathcal{U}$, the expected cost is given by
\begin{equation}\label{simplified_expected_cost}
    J\big(\pi, \Pi^u_\pi(\cdot), \tau\big)
    =
    \mathbb{E}^u \left[
        g\big(\Pi^u_\pi(\tau)\big)
        +
        (\varphi(u) + c) \cdot \tau
    \right].
\end{equation}

Taking yet another leap of faith, we posit that we may restrict to stopping times $\tau$ that mark the first exit time of the posterior process $\Pi^u_\pi(\cdot)$ from an interval $(\delta, 1 - \delta)$, i.e., $\tau = \tau_\delta$ with
\begin{equation}\label{def_hitting_time}
    \tau_\delta \coloneqq \inf\left\{t \ge 0 : \Pi^u_\pi(t) \notin (\delta, 1 - \delta)\right\}.
\end{equation}
This postulate is reasonable when the control $u(\cdot) \equiv u$ is constant, as it reduces \eqref{def_value_function} to the sequential testing problems of \cite{Shiryaev67} and \cite{CampZha24}, in which the optimal stopping rule takes precisely this form for some optimal boundary $\delta$.

Fixing such a stopping rule, for $\pi\in(\delta,1-\delta)$ the cost \eqref{simplified_expected_cost} becomes
\begin{equation}\label{simplified_expected_cost_2}
    \begin{split}
        J\big(\pi,\Pi^u_\pi(\cdot), \tau_\delta\big)
        &=
        g(\delta) + 
        (\varphi(u) + c) \cdot
        \ex^u\big[
            \tau_\delta
        \big]
        \\&=
        g(\delta) + 
        \frac{2(\varphi(u) + c)}{u^2} \left(
            (1 - 2\pi) \ln\left(\frac{\pi}{1-\pi}\right) - (1 - 2\delta) \ln\left(\frac{\delta}{1-\delta}\right)\right),
    \end{split}
\end{equation}
on the strength of the following proposition whose proof is located in Appendix \ref{subsec_proof_of_hitting_time_prop}. On the other hand, if $\pi \notin (\delta, 1 - \delta)$, then $\mathbb{E}^u[\tau_\delta] = 0$ and $J\big(\pi, \Pi^u_\pi(\cdot), \tau_\delta\big) = g(\pi)$. We remark that the value functions of \cite{Shiryaev67} and \cite{CampZha24} do indeed match the form of \eqref{simplified_expected_cost_2} for suitably chosen parameters.

\begin{Prop}\label{prop_hitting_time_expect}
    Fix any $\delta\in(0,1/2)$ and $\pi \in (\delta, 1-\delta)$. Consider the process $\Pi^u_\pi(\cdot)$ of \eqref{controlled_dynamicccs_pi} with constant control $u(\cdot) \equiv u$ for some fixed $u \in \cu$, and let $\tau_\delta$ be defined as in \eqref{def_hitting_time}. Then,
    \begin{equation}\label{delta_hittinig_expectation}
        \ex^u \big[ \tau_\delta \big] = \frac{2}{u^2} \left((1 - 2\pi) \ln\left(\frac{\pi}{1-\pi}\right) - (1 - 2\delta) \ln\left(\frac{\delta}{1-\delta}\right)\right).
    \end{equation}
\end{Prop}

By the form of the expectation in \eqref{delta_hittinig_expectation}, we observe that
\begin{equation}
(1 - 2\pi) \ln\left(\frac{\pi}{1 - \pi}\right) - (1 - 2\delta) \ln\left(\frac{\delta}{1 - \delta}\right) \ge 0, \quad \forall\, \pi \in (\delta, 1 - \delta).
\end{equation}
In view of \eqref{simplified_expected_cost_2}, this suggests that the optimal control $u^*$ should minimize the function
\begin{equation}\label{definition_eta_function}
    \eta(u) \coloneqq \frac{\varphi(u) + c}{u^2}.
\end{equation}
Whereas, if the function $\eta(\cdot)$ does not attain its infimum, i.e., if $M \coloneqq \inf_{u \in \mathcal{U}} \eta(u)$ satisfies $M < \eta(u)$ for all $u \in \mathcal{U}$, then it is natural to expect that the optimal control process does not exist. Indeed, for any admissible $u(\cdot)$, one could always construct an alternative control that more closely approximates $M$, leading to strict improvement in cost.

As it turns out, the conjectures above do in fact capture the true structure of the solution. Before presenting a rigorous justification, we introduce some additional notation and prove a key auxiliary result. For convenience of exposition, define the function
\begin{equation}\label{def_psi_function}
    \Psi(\pi) \coloneqq (1 - 2\pi) \ln\left(\frac{\pi}{1 - \pi}\right), \quad 0 < \pi < 1,
\end{equation}
whose derivatives are given by
\begin{equation}\label{psi_derivatives}
    \Psi'(\pi) = \frac{1 - 2\pi}{\pi(1 - \pi)} - 2 \ln\left( \frac{\pi}{1 - \pi} \right), 
    \qquad
    \Psi''(\pi) = -\frac{1}{\pi^2 (1 - \pi)^2}.
\end{equation}
In particular, $\Psi(\cdot)$ is symmetric about $1/2$, concave, and non-positive on $(0,1)$. For each fixed constant $K > 0$, we also define the function
\begin{equation}\label{def_H_function}
    H(\pi) = H(\pi; K) \coloneqq g(\pi) - K \Psi(\pi),
\end{equation}
which plays a central role in what follows. The next lemma, adapted from \cite{CampZha24}, is crucial to the analysis in the next section. 

\begin{Lem}[Corollary 2.4 of Campbell and Zhang '25]\label{lem_properties_of_H}
    Fix $K > 0$, and consider the function $H(\cdot) = H(\,\cdot\,; K)$ defined in \eqref{def_H_function}. There exists $\pi^* = \pi^*(K) \in (0, 1/2]$ such that $H(\cdot)$ is strictly convex on $(0, \pi^*) \cup (1 - \pi^*, 1)$ and strictly concave on $(\pi^*, 1 - \pi^*)$. In particular, if
    \begin{itemize}
        \item[(a)] $g(\pi) \propto \pi \wedge (1 - \pi)$, or
        \item[(b)] $g''(1/2) \ge -16K$,
    \end{itemize}
    then $\pi^* = 1/2$. Otherwise, $\pi^* < 1/2$. Moreover, in case (b), $H(\cdot)$ is strictly convex on the entire interval $(0, 1)$.
\end{Lem}

\begin{proof}
Case (a) follows directly from the strict concavity of $\Psi(\cdot)$. Since $g(\cdot)$ is piecewise affine, by directly differentiating, $H''(\pi) = -K \Psi''(\pi) > 0$ for all $\pi \in (0,1) \setminus \{1/2\}$.

For more general $g(\cdot)$ satisfying Assumption~\ref{ass_1}(b), the result follows from \cite[Corollary~2.4]{CampZha24} and the surrounding discussion. While we refer the reader to that work for full details, we briefly outline the main idea. The convexity and concavity of $H(\cdot)$ is governed by the sign of $H''(\cdot)$, or equivalently, of the function
\begin{equation}
 g''(\pi)\, \pi^2(1 - \pi)^2 + K.
\end{equation}
Under Assumption~\ref{ass_1}(b), this expression is strictly decreasing on $(0, 1/2)$ (and increasing on $(1/2, 1)$), so the sign of $H''(\pi)$ can change at most once on each interval. Moreover, since $g(0) = g(1) = 0$, it follows from \cite[Lemma~2.2]{CampZha24} that
\begin{equation}
    \lim_{\pi \downarrow 0} g''(\pi)\, \pi^2 (1 - \pi)^2 = \lim_{\pi \uparrow 1} g''(\pi)\, \pi^2 (1 - \pi)^2 = 0.
\end{equation}
Thus, $H''(\pi) > 0$ near $0$ and $1$, so $H(\cdot)$ is strictly convex near the endpoints. On the other hand, $H''(1/2) < 0$ if and only if $g''(1/2) < -16K$, in which case $H(\cdot)$ is strictly concave in a neighborhood of $1/2$. The result follows.
\end{proof}

We record here a few consequences of the preceding result for later reference. 

\begin{Cor}\label{cor_H_function}
    Fix any constant $K > 0$. The equation $g'(\pi) = K\Psi'(\pi)$, equivalently, $H'(\pi; K) = 0$, has at most one solution $\pi_0 = \pi_0(K)<\pi^*(K)$ in the interval $(0, 1/2)$. A solution exists if and only if either
    \begin{itemize}
        \item[(a)] $g(\pi) \propto \pi \wedge (1 - \pi)$, or
        \item[(b)] $g(\cdot) \in C^2\big((0, 1)\big)$ and $ g''(1/2) < -16K$.
    \end{itemize}
    In this case, $H(\,\cdot\,; K)$ attains its minimum at $\pi_0$ and $1 - \pi_0$.
\end{Cor}

\begin{proof}
By Lemma~\ref{lem_properties_of_H}, if neither (a) nor (b) holds, then $H(\,\cdot\,; K)$ has a unique minimum. By symmetry, this minimum must occur at $\pi = 1/2$, implying that $H'(\,\cdot\,; K)$ does not vanish on $(0, 1/2)$.

Now suppose that (a) or (b) holds. Lemma~\ref{lem_properties_of_H} asserts that $H(\cdot)$ has at most three critical points in $(0,1)$: namely, $\pi_0 \in (0, \pi^*]$, $1/2$, and $1 - \pi_0 \in [1 - \pi^*, 1)$. Under (a), $g'(\pi)$ is a positive \emph{constant} on $(0,1/2)$. Since $\Psi'(\cdot)$ is positive and strictly decreasing on $(0,1/2)$ with $\Psi'(1/2) = 0$ and $\lim_{\pi\downarrow 0} \Psi'(\pi) = \infty$, there is a unique solution $\pi_0(K)<\pi^*(K)=1/2$ to $g'(\pi)=K\Psi'(\pi)$ for any $K>0$. Under (b), since $H(\cdot)$ is a $C^2(0,1)$ function that is strictly convex on $(0, \pi^*)$ and $(1 - \pi^*, 1)$, and strictly concave on $(\pi^*, 1 - \pi^*)$, the function admits local minima at some $\pi_0\in[0,\pi^*)$ and $1 - \pi_0\in(1-\pi^*,1]$. As $g(\cdot)$ is bounded and $g(0) = g(1) = 0$, we have $\lim_{\pi \downarrow 0} H(\pi) = \lim_{\pi \uparrow 1} H(\pi) = \infty$, so no minimizer occurs at these outer boundaries. Then, by symmetry, these must also be global minima on $(0,1)$. The result follows.
\end{proof}

\section{Main Results}\label{sec_main_results}

We collect all of our main results up front in a single theorem. The remainder of this section will build up to its proof. 

\begin{Th}\label{thm_main}
Consider the stochastic control problem with discretionary stopping defined in \eqref{def_value_function}. Let $\eta(\cdot)$ be as in \eqref{definition_eta_function}, and define $M \coloneqq \inf_{u \in \mathcal{U}} \eta(u) < \infty$. For $M > 0$, let $A^*=A^*(M)$ denote the smallest solution in $(0, 1/2]$ to the equation $g'(\pi) = 2M \Psi'(\pi)$; otherwise, set $A^* = 0$.

\smallskip

\noindent The value function admits the following three-part characterization:
\begin{enumerate}
    \item[(1)] If $M > 0$, then the value function is given by $V(\pi) = \widehat{V}(\pi; 2M, A^*)$,
    where
    \begin{equation}\label{value_function_std}
        \widehat{V}(\pi; \alpha, \delta) \coloneqq
        \begin{cases}
            \alpha \cdot \big(\Psi(\pi) - \Psi(\delta)\big) + g(\delta), & \pi \in (\delta, 1 - \delta),\\
            g(\pi), & \pi \notin (\delta, 1 - \delta).
        \end{cases}
    \end{equation}
    
    \item[(2)] If $M = 0$, the value function is identically zero: $V(\pi) \equiv 0$ for all $\pi \in [0, 1]$.
    
    \item[(3)] If $M < 0$, the value function is identically $-\infty$: $V(\pi) \equiv -\infty$ for all $\pi \in [0, 1]$.
\end{enumerate}

\smallskip

\noindent Moreover, the optimal policies exhibit the following structure:
\begin{itemize}
    \item If $M\geq0$ and the infimum of $\eta(\cdot)$ is attained at some $u^* \in \mathcal{U}$, then the constant control $u^*(t) \equiv u^*$ is optimal, as is the stopping time $\tau_{A^*}$ defined in \eqref{def_hitting_time}.

    \item If $M\geq0$ and the infimum of $\eta(\cdot)$ is not attained, then no optimal control-stopping pair exists for initial beliefs $\pi \in (A^*, 1 - A^*)$. On the other hand, for $\pi \notin (A^*, 1 - A^*)$, an optimal stopping time is $\tau_{A^*} \equiv 0$, and therefore any control is optimal.

    \item If $M<0$, then any constant control $u^*(t) \equiv u^*$ satisfying $\eta(u^*)<0$ with $\tau^*\equiv \infty$ is optimal.
\end{itemize}

\smallskip

\noindent Irrespective of attainment, for any $\varepsilon > 0$ there exists an $\varepsilon$-optimal strategy consisting of a constant control $u(t) \equiv \mathfrak{u}\in\mathcal{U}$ and a stopping time $\tau_\delta$ of the form \eqref{def_hitting_time}, where $\mathfrak{u}$ is chosen along a minimizing sequence for $\eta(\cdot)$ and $\delta > 0$ is sufficiently close to $A^*$.
\end{Th}

\begin{remark}
    By an $\eps$-optimal strategy we mean a pair $(u(\cdot), \tau)$ such that the expected cost $J(\pi, \Pi^u_\pi(\cdot), \tau)$ satisfies $J(\pi, \Pi^u_\pi(\cdot), \tau) - V(\pi) < \eps$ for each $\pi \in [0, 1]$. The precise construction of the $\eps$-optimal controls is provided in the proof.
\end{remark}

\begin{remark}
    If $A^* = 0$ and $\pi \in (0, 1)$, then $\tau_{A^*} = \infty$ almost surely by applying Feller's test for explosions (cf. \cite[Theorem 5.5.29]{BMSC}) to $\Pi^{u}_\pi(\cdot)$ for any constant control $u(\cdot)\equiv u\in\mathcal{U}$. On the other hand, if $A^*=1/2$ then $\tau_{A^*}\equiv0$. 
\end{remark}

The first step in proving Theorem~\ref{thm_main} is to show that the variational inequalities \ref{var_ineq_1}--\ref{var_ineq_3} for the value function, found at the beginning of Section \ref{sec_prem_analysis}, are satisfied by $\widehat{V}(\cdot)$ of \eqref{value_function_std}. This is established in the proposition below.

\begin{Prop}\label{prop_var_inequalities}
     Suppose $M := \inf_{u \in \mathcal{U}} \eta(u) \ge 0$, and fix the notation of Theorem~\ref{thm_main}.
    \begin{enumerate}
        \item[(a)] If $M=0$, then the constant function $\widehat{V}(\cdot)\equiv0$ satisfies the variational inequalities \ref{var_ineq_1}--\ref{var_ineq_2}.
        \item[(b)] If $M > 0$, then the function $\widehat{V}(\cdot) = \widehat{V}(\cdot\,; 2M, A^*)$ defined in \eqref{value_function_std} belongs to the class $C^1\big((0, 1)\big) \cap C^2\big((0, 1) \setminus \{A^*, 1 - A^*\}\big)$ and satisfies the variational inequalities~\ref{var_ineq_1}--\ref{var_ineq_2}.        
    \end{enumerate}
    Moreover, if a bounded minimizing sequence for $\eta(\cdot)$ exists, then~\ref{var_ineq_3} also holds in both (a) and (b).
\end{Prop}

\begin{proof}
    We treat each case separately. 
    
    \underline{Case (a):} Since $g(0)=g(1)=0$, by concavity $g(\cdot)\geq0$, and so inequality \ref{var_ineq_1} is immediate. Since $M=0$ and $u^2>0$, $\varphi(u)+c\geq0$ for all $u\in\mathcal{U}$ by the definition of the infimum.  The non-negativity in \ref{var_ineq_2} is then clear.
    
    Suppose further that a bounded minimizing sequence $(u_n)_{n\geq1}$ taking values in $\mathcal{U}$ exists. Then, there exists $C>0$ such that $|u_n|\leq C$ for all $n$. For each $n\geq1$,
    \begin{equation}
    0\leq \varphi(u_n)+c
    = u_n^2\eta(u_n)
    \le C^2\eta(u_n)\to 0.
    \end{equation}
    We conclude $\inf_{u\in\mathcal{U}}\{\varphi(u)+c\} =0$ which, coupled with the fact that $\widehat{V}''(\cdot)\equiv 0$, directly gives \emph{equality} in \ref{var_ineq_2}--\ref{var_ineq_3} for all $\pi$.

    \underline{Case (b):}
    If $M > 0$, the regularity of $\widehat{V}(\cdot)$ is easily checked from the representation \eqref{value_function_std} and the choice of $A^*$. Thus, we only need to show that $\widehat{V}(\cdot)$ satisfies the variational inequalities.

    To verify that $\widehat{V}(\cdot)$ satisfies \ref{var_ineq_1}, consider first the function $H(\cdot) = H(\,\cdot\,; 2M)$ defined in \eqref{def_H_function}. By the definition of $A^*$ and Corollary~\ref{cor_H_function}, $H(\cdot)$ attains its minimum at the points $A^*$ and $1 - A^*$. As a result, for all $\pi \in (A^*, 1 - A^*)$, we have
    \begin{equation}
        g(\pi) - \widehat{V}(\pi) = g(\pi) - 2M\Psi(\pi) + 2M \Psi(A^*) - g(A^*) = H(\pi) - H(A^*) \ge 0.
    \end{equation}
    On the other hand, for $\pi \in [0, A^*] \cup [1 - A^*, 1]$, the definition of $\widehat{V}(\cdot)$ implies $\widehat{V}(\pi) = g(\pi)$, so inequality~\ref{var_ineq_1} holds trivially. In particular, the complementarity condition~\ref{var_ineq_3} is also satisfied on this set, since we have equality in \ref{var_ineq_1}.   

    To verify that $\widehat{V}(\cdot)$ satisfies the variational inequality~\ref{var_ineq_2}, we consider separately the cases $\pi \in (A^*, 1 - A^*)$ and $\pi \in [0, A^*] \cup [1 - A^*, 1]$. Since $\widehat{V}''(\cdot)$ may fail to exist at the boundaries $\{A^*,1-A^*\}$ and the
    endpoints $\{0,1\}$, we restrict our attention to points in the interior of these intervals.

    For $\pi \in (A^*, 1 - A^*)$, we compute the second derivative of $\widehat{V}(\cdot)$ using~\eqref{psi_derivatives}. Substituting this into inequality~\ref{var_ineq_2}, the condition becomes
    \begin{equation}\label{eq:var_ineq_reduction_cont_reg}
        -M u^2 + \varphi(u) + c \ge 0, \quad \forall \ u \in \mathcal{U},
    \end{equation}
    or, equivalently,
    \begin{equation}
        \frac{\varphi(u) + c}{u^2} \ge M, \quad \forall \ u \in \mathcal{U}.
    \end{equation}
    This, however, is a direct consequence of the definition $M = \inf_{u \in \mathcal{U}} \eta(u)$.

    As for $\pi \in (0, A^*) \cup (1 - A^*, 1)$, we must verify that
    \begin{equation}
    \frac{1}{2} g''(\pi)\, \pi^2 (1 - \pi)^2 u^2 + \varphi(u) + c \ge 0, \quad \forall\, u \in \mathcal{U},
    \end{equation}
    which is equivalent to
    \begin{equation}\label{nec_eq_for_g}
    g''(\pi) + \frac{2(\varphi(u) + c)}{\pi^2 (1 - \pi)^2 u^2} \ge 0, \quad \forall\, u \in \mathcal{U}.
    \end{equation}
    Using the definition $M = \inf_{u \in \mathcal{U}} \eta(u)$, we obtain the lower bound
    \begin{equation}\label{nec_eq_for_g_2}
    g''(\pi) + \frac{2(\varphi(u) + c)}{\pi^2 (1 - \pi)^2 u^2} \ge g''(\pi) + \frac{2M}{\pi^2 (1 - \pi)^2} = H''(\pi), \quad \forall\, u \in \mathcal{U}.
    \end{equation}
    Since $H(\cdot)$ is convex on $(0, \pi^*) \cup (1 - \pi^*, 1)$ and $A^* < \pi^*$ by Corollary~\ref{cor_H_function}, it follows that $H''(\pi) \ge 0$ for all $\pi \in (0, A^*) \cup (1 - A^*, 1)$. The desired inequality \eqref{nec_eq_for_g} then follows directly from \eqref{nec_eq_for_g_2}.

    To complete the proof, suppose once more that a bounded minimizing sequence $(u_n)_{n \geq 1}$ taking values in $\mathcal{U}$ exists. Since condition~\ref{var_ineq_3} has already been verified for $\pi \in [0, A^*] \cup [1 - A^*, 1]$, it remains to establish it for $\pi \in (A^*, 1 - A^*)$. To this end, we revisit \eqref{eq:var_ineq_reduction_cont_reg}, which corresponds to \ref{var_ineq_2} on this interval. Since the sequence $(u_n)$ is bounded, there exists a constant $C > 0$ such that $|u_n| \leq C$ for all $n \geq 1$. Using that $\eta(u_n) \ge M$, we obtain
    \begin{equation}
    0 \le -M u_n^2 + \varphi(u_n) + c = u_n^2 \left( \eta(u_n) - M \right) \le C^2 \left( \eta(u_n) - M \right) \to 0.
    \end{equation}
    It follows that $\inf_{u \in \mathcal{U}} \left\{ -M u^2 + \varphi(u) + c \right\} = 0$, and hence \emph{equality} holds in \ref{var_ineq_2} on $(A^*,1-A^*)$. Accordingly, condition~\ref{var_ineq_3} is satisfied on all of $(0,1)$.
\end{proof}

\begin{remark}
    If the minimum of $\eta(\cdot)$ over $\mathcal{U}$ is attained at some $u^*\in\mathcal{U}$, then a bounded minimizing sequence is trivially constructed by choosing the constant sequence $u_n\equiv u^*$, $n\geq0$. Moreover, if $\mathcal{U}$ is bounded the existence of such a sequence is also immediate regardless of attainment.
\end{remark}

We are now ready to prove Theorem \ref{thm_main}.

\subsection{Proof of Theorem \ref{thm_main}}

\begin{proof}

    We will prove each case in turn.

    \underline{Case (1), $M > 0$:} First, we show that the suggested function $\widehat{V}(\cdot) = \widehat{V}(\, \cdot \,, 2M, A^*)$ of \eqref{value_function_std} indeed provides a lower bound on the expected cost. Consider any admissible controlled process $\Pi^u_\pi(\cdot) \in \ca(\pi), \, \pi \in [0, 1]$ with control $u(\cdot)$, and let $\tau$ be any $\mathbb{F}^u$--measurable stopping time. Note that we can assume 
    \begin{equation}\label{basic_admissinility}
        \ex^u\left[\int_0^\tau \varphi(u(s)) \, ds + c\tau\right] = \ex^u\left[\int_0^\tau \big(\varphi(u(s)) + c \big) \, ds \right] < \infty,
    \end{equation}
    since otherwise the expected cost \eqref{def_expected_loss_pi} is automatically infinite, and such a stopping time cannot be optimal because the cost $g(\pi)$ associated with immediate stopping $\tau \equiv 0$ is finite.

    If $\pi = 0$ or $1$, the process $\Pi^u_\pi(\cdot) \equiv \pi$, so it is clear that $\widehat{V}(\pi) = g(\pi)$ provides a lower bound on the expected cost. Thus, we may assume $\pi \in (0, 1)$.
    In this case, consider the process $\widehat{V}(\Pi^u_\pi(\cdot))$. By Proposition~\ref{prop_var_inequalities}, the candidate value function $\widehat{V}(\cdot)$ is piece-wise $C^2$, with the possible exceptions arising at $A^*$ and $1-A^*$. However, it is readily checked that the left and right second derivatives exist and are bounded in a neighborhood of these points. As a consequence, $\widehat{V}(\cdot)$ has a locally Lipschitz (and hence locally absolutely continuous) derivative. Therefore, on the strength of Itô's rule in Problem 7.3, p. 219 of \cite{BMSC}, we obtain
    \begin{equation}\label{ito_rule}
        \begin{split}
            \widehat{V}(\Pi^u_\pi(\tau)) 
            =
            \widehat{V}(\pi) &+ \int_0^\tau \widehat{V}'(\Pi^u_\pi(s)) \, d\Pi^u_\pi(s) + \frac{1}{2}\int_0^\tau \widehat{V}''(\Pi^u_\pi(s)) \, d\langle\Pi^u_\pi\rangle(s)
            \\=
            \widehat{V}(\pi) &+ \frac{1}{2}\int_0^\tau \widehat{V}''(\Pi^u_\pi(s)) \big(u(s) \, \Pi^u_\pi(s) \, (1 - \Pi^u_\pi(s))\big)^2 \, ds
            \\&
            + \int_0^\tau \widehat{V}'(\Pi^u_\pi(s)) \, u(s) \, \Pi^u_\pi(s) \, (1-\Pi^u_\pi(s)) \, dB^u(s).
        \end{split}
    \end{equation}
    
    Note that the function $\pi \mapsto \widehat{V}'(\pi)\pi (1 - \pi)$ is bounded. On $(A^*, 1-A^*)$, it follows from the equality $\widehat{V}'(\cdot) = 2M \Psi'(\cdot)$ and \eqref{psi_derivatives}, while on the complement set it becomes a statement about the decay rate of the function $g(\cdot)$. Under Assumption \ref{ass_1}(a) this bound is obvious, while the case of Assumption \ref{ass_1}(b) follows from Remark 2.3 of \cite{campbell2024mean}. Thus, the stochastic integral 
    \begin{equation}\label{stoch_integral_in_ito}
        N_\pi^u(t) \coloneqq \int_0^{t \wedge \tau} \widehat{V}'(\Pi^u_\pi(s)) \, u(s) \, \Pi^u_\pi(s) \, (1-\Pi^u_\pi(s)) \, dB^u(s), \quad 0 \le t < \infty
    \end{equation}
    has quadratic variation
    \begin{equation}\label{quadratic_variation_bound}
        \begin{split}
            \langle N^u_\pi \rangle(t) 
            &=
            \int_0^{t \wedge \tau} \Big\{\widehat{V}'(\Pi^u_\pi(s)) \, u(s) \, \Pi^u_\pi(s) \, (1-\Pi^u_\pi(s))\Big\}^2 ds
            \\&\le
            \int_0^{t \wedge \tau} D^2 \, u^2(s) \, ds
            \le
            \int_0^{t \wedge \tau} \frac{D^2}{M} \, \big(\varphi(u(s)) + c\big) \, ds,
        \end{split}
    \end{equation}
    which is integrable at infinity, i.e., $\ex^u[\langle N_\pi^u \rangle(\infty)] < \infty$, due to \eqref{basic_admissinility}. Here, $D$ is the bound on $\pi \mapsto \widehat{V}'(\pi)\pi (1 - \pi)$, and the last inequality in \eqref{quadratic_variation_bound} follows from the definition of $M$. As a result, the process $N_\pi^u(\cdot)$ of \eqref{stoch_integral_in_ito} is a 
    uniformly integrable martingale with a last element
    (cf. \cite[Problems 1.3.20 and 1.5.24]{BMSC}). 
    
    By taking expectations on both sides of \eqref{ito_rule} and using the optional sampling theorem (cf. \cite[Theorem 1.3.22]{BMSC}),
    we obtain
    \begin{equation}\label{inequality_for_V_Ito}
        \begin{split}
            \widehat{V}(\pi) 
            &= 
            \ex^u\left[
                \widehat{V}(\Pi^u_\pi(\tau)) - 
                \frac{1}{2}\int_0^\tau \widehat{V}''(\Pi^u_\pi(s)) \big(u(s) \, \Pi^u_\pi(s) \, (1 - \Pi^u_\pi(s))\big)^2 \, ds
            \right]
            \\&\le
            \ex^u\Big[
                g(\Pi^u_\pi(\tau)) + \int_0^\tau \big(\varphi(u(s)) + c \big) \, ds 
            \Big] = J(\pi, \Pi^u_\pi(\cdot), \tau),
        \end{split}
    \end{equation}
    where the inequality follows from the fact that $\widehat{V}(\cdot)$ satisfies the variational inequalities \ref{var_ineq_1}--\ref{var_ineq_2}. Therefore, $\widehat{V}(\cdot)$ indeed provides a lower bound on the smallest expected cost, and it remains to discuss its attainability.

    As before, if $\pi = 0$ or $1$, there is nothing to prove, so we assume $\pi \in (0, 1)$. Suppose first that $M=\inf_{u\in\mathcal{U}}\eta(u)$ is attained at some $u^*\in\mathcal{U}$. If $\pi\not\in(A^*,1-A^*)$, then $\tau_{A^*}\equiv 0$ and for any constant control $u\in\mathcal{U}$ we get \begin{equation}V(\pi)\leq J\big(\pi,\Pi^{u}_\pi(\cdot),\tau_{A^*}) =  g(\pi)=\widehat{V}(\pi).\end{equation} Consider instead the case $\pi\in(A^*,1-A^*)$. Fix the constant control process $u^*(\cdot) \equiv u^*$ and choose the stopping time $\tau_{A^*}$ of \eqref{def_hitting_time}. By \eqref{simplified_expected_cost_2} and the definition of $\Psi(\cdot)$ in \eqref{def_psi_function} we have
    \begin{align*}
        V(\pi) &\leq J\big(\pi,\Pi^{u^*}_\pi(\cdot), \tau_{A^*}\big)=
        g(A^*) + 
        \frac{2(\varphi(u^*) + c)}{(u^*)^2} \left(
            \Psi(\pi) - \Psi(A^*)\right)\\
            &=g(A^*) + 
       2M\left(
            \Psi(\pi) - \Psi(A^*)\right)=\widehat{V}(\pi).
    \end{align*}
    Coupled with the lower bound, this gives $V(\cdot) \equiv \widehat{V}(\cdot)$. Moreover, this readily proves optimality of the selected pair $(u^*(\cdot), \tau_{A^*})$.

    Next, suppose that $M$ is not attained. The case $\pi\not\in(A^*,1-A^*)$ is unchanged. Therefore, take $\pi\in(A^*,1-A^*)$ and assume without loss of generality that $A^* < 1/2$. Let $(u_n)_{n\geq0}$ be a $\mathcal{U}$--valued minimizing sequence for $\eta(\cdot)$.  Similarly, let $(\delta_n)_{n\geq0}$ be a sequence in $(0,1/2)$ such that $\delta_n\downarrow A^*$. For all sufficiently large $n$, we must have $\delta_n<\pi<1-\delta_n$ by the hypothesis. Thus, for large enough $n$, by setting $u_n(\cdot)\equiv u_n$ and invoking \eqref{simplified_expected_cost_2} we obtain
    \begin{equation}
        V(\pi)\leq J\big(\pi,\Pi^{u_n}_\pi(\cdot), \tau_{\delta_n}\big)=
        g(\delta_n) + 
        \frac{2(\varphi(u_n) + c)}{u_n^2} \left(
            \Psi(\pi) - \Psi(\delta_n)\right).
    \end{equation}
    Letting $n\to\infty$ we find
    \begin{equation}V(\pi) \leq g(\delta_n) + 
        \frac{2(\varphi(u_n) + c)}{u_n^2} \left(
            \Psi(\pi) - \Psi(\delta_n)\right)\to g(A^*) + 
        2M \left(
            \Psi(\pi) - \Psi(A^*)\right) =\widehat{V}(\pi)\end{equation}
    as required. Moreover, for any $\eps>0$ we may take $n$ sufficiently large so that the pair $(u_n(\cdot),\tau_{\delta_n})$ defines a $\varepsilon$-optimal strategy. 

    Finally, we show that if $M$ is not attained and $\pi\in(A^*,1-A^*)$ then no optimal control exists. To this end, fix any pair $(u(\cdot), \tau)$. Since $V(\cdot)<g(\cdot)$ in this region, we may without loss of generality consider only stopping times for which $\{\tau>0\}$ has positive measure. Moreover, we may further restrict to control pairs $(u(\cdot),\tau)$ satisfying \eqref{basic_admissinility}. As we have already identified $\widehat{V}(\cdot)\equiv V(\cdot)$, by the same arguments as above we may replace $\widehat{V}(\cdot)$ with $V(\cdot)$ in \eqref{inequality_for_V_Ito} to get
    \begin{equation}
        \begin{split}
            V(\pi) 
            &= 
            \ex^u\left[
                V(\Pi^u_\pi(\tau)) - 
                \frac{1}{2}\int_0^\tau V''(\Pi^u_\pi(s)) \big(u(s) \, \Pi^u_\pi(s) \, (1 - \Pi^u_\pi(s))\big)^2 \, ds
            \right]\\
            &< \ex^u\left[
                g(\Pi^u_\pi(\tau)) + \int_0^\tau (\varphi(u(s)) +c)\, ds
            \right] = J(\pi, \Pi^u_\pi(\cdot), \tau).
        \end{split}
    \end{equation}
    Here we use that the variational inequality \ref{var_ineq_2} is strict in the continuation region at any $u(s)\in\mathcal{U}$ in conjunction with the requirement $\mathbb{P}^u(\tau>0)>0$. To see the former, we may substitute $V(\cdot)\equiv\widehat{V}(\cdot)$ into \ref{var_ineq_2} to obtain \eqref{var_ineq_2_simplified} with $\alpha=2M$ in the continuation region. If equality holds in \eqref{var_ineq_2_simplified} at some $u(s)\in\mathcal{U}$ then $\eta(\cdot)$ attains $M$ at $u(s)$ which contradicts our assumptions. Since $\pi\in(A^*,1-A^*)$ and $\tau>0$, $\Pi_\pi^u(\cdot)$ necessarily spends strictly positive time in the continuation region. Combining these facts ensures that 
    \[-\frac{1}{2}V''(\Pi^u_\pi(s)) \big(u(s) \, \Pi^u_\pi(s) \, (1 - \Pi^u_\pi(s))\big)^2<(\varphi(u(s)) +c)\]
    on a set of $dt\otimes d\mathbb{P}^u$ positive measure, which is precisely what yields the desired strict inequality in the argument above. Since $(u(\cdot),\tau)$ were chosen arbitrarily, the proof is complete.

    \underline{Case (2), $M = 0$:} Fix $\pi \in [0, 1]$. By the definition of the infimum $\eta(u)\geq0$ for all $u\in\mathcal{U}$. In particular, this implies that $\varphi(u)+c\geq0$ for all such $u$. Since $g(\cdot)\geq0$, we readily conclude that $V(\cdot)\geq0$. Moreover, if $\pi = 0$ or $1$, immediate stopping guarantees the expected cost $g(0) = g(1) = 0$, which finishes the proof for this case. We thus assume from now onwards that $\pi \in (0, 1)$.

    If $M$ is attained at some $u^*\in\mathcal{U}$, we may set $u^*(\cdot)\equiv u^*$. Then, for all stopping times $\tau$,
    \begin{equation}
    V(\pi) \leq \mathbb{E}^u \left[
        g\big(\Pi^{u^*}_\pi(\tau)\big)
        + \int_0^\tau (\varphi(u(t)) + c) \, dt
    \right]=\mathbb{E}^u \left[
        g\big(\Pi^{u^*}_\pi(\tau)\big)\right].\end{equation}
    If we choose $\tau_{A^*}$ with $A^*=0$, by Feller's test for explosions (cf. \cite[Theorem 5.5.29]{BMSC}) $\tau_{A^*}\equiv\infty$. At the same time, recalling from Lemma \ref{lem:pi.consistent} that $\lim_{t\to\infty}\Pi_\pi^u(t)\in\{0,1\}$ a.s.,\ we conclude $g(\Pi_\pi^{u^*}(\tau_{A^*}))= \lim_{t\to\infty}g(\Pi_\pi^{u^*}(t))=0$ a.s. It follows that $V(\cdot) \equiv 0$.

    If $M$ is not attained, then we may carefully modify the argument employed in Case (1). We again define $(u_n)_{n\geq0}$ to be a $\mathcal{U}$--valued minimizing sequence for $\eta(u)$ and set $u_n(\cdot)\equiv u_n$. For any fixed $\pi\in(0,1)$ and any $\delta>0$ satisfying $\delta<\pi<1-\delta$ we obtain from \eqref{simplified_expected_cost_2} that
    \begin{equation}
        V(\pi)\leq J\big(\pi,\Pi^{u_n}_\pi(\cdot), \tau_{\delta}\big)=
        g(\delta) + 
        \frac{2(\varphi(u_n) + c)}{u_n^2} \left(
            \Psi(\pi) - \Psi(\delta)\right).
    \end{equation}
    Observe that for any fixed $\pi$, 
    $\Psi(\pi)-\Psi(\delta)\to\infty$
    as $\delta\downarrow 0$. Since $\Psi(\cdot)\leq0$ is continuous and strictly increasing on $(0,1/2)$ there exists a unique solution $\delta_k\in(0,1/2)$ to the equation
    \begin{equation}\label{eqn:psi.delta}
        \Psi(\pi) - \Psi(\delta_k) = k 
    \end{equation}
    for all sufficiently large $k\in\mathbb{N}$. Let $k^*$ be the smallest $k$ such that a solution to \eqref{eqn:psi.delta} exists and consider the corresponding sequence $(\delta_k)_{k\geq0}$ with the convention that $\delta_k=1/2$ if $k\leq k^*$. By construction, $\delta_{k}\to 0$. Next, observe that since $\lim_{n\to\infty}\eta(u_n)=0$ we may choose a subsequence $(u_{n_k})_{k\geq0}$ such that $\eta(u_{n_k})\leq k^{-2}$. Fixing $\pi$ and making the choice $u_{n_k}(\cdot) \equiv u_{n_k}$ and $\tau_{\delta_k}$, we obtain, for all sufficiently large $k$,
    \begin{equation}
        V(\pi)\leq J\big(\pi,\Pi^{u_{n_k}}_\pi(\cdot), \tau_{\delta_k}\big)=
        g(\delta_k) + 
        2k\eta(u_{n_k})\leq g(\delta_k) + 2k^{-1} \to 0
    \end{equation}
    as $k\to\infty$ since $g(0)=0$. We note once more that this choice of sequence can be used to construct $\varepsilon$-optimal strategies.

    In fact, we can again show that when $M$ is not attained and $\pi\in(0,1)$ there is no optimal control. Indeed, since $V(\cdot)<g(\cdot)$, we conclude that $\tau\equiv 0$ cannot be optimal. Thus, it suffices to restrict to stopping times for which the set $\{\tau>0\}$ has positive measure. At the same time, by the definition of non-attainment, $\varphi(u)+c>0$ for all $u\in\mathcal{U}$. Then, for any admissible $u(\cdot)$
    \begin{equation}J\big(\pi,\Pi^{u}_\pi(\cdot), \tau\big)=\mathbb{E}^u\left[g(\Pi^u_\pi(\tau))+\int_0^\tau (\varphi(u(t)) + c) dt\right]\geq \mathbb{E}^u\left[\int_0^\tau (\varphi(u(t)) + c) dt\right]>0=V(\pi).\end{equation}
    This completes the proof.

    \underline{Case (3), $M < 0$:} We trivially have $V(\cdot)\geq -\infty$. Since $\inf_{u\in\mathcal{U}}\eta(u)<0$, regardless of attainment, there necessarily exists $\delta>0$ and $\mathfrak{u}\in\mathcal{U}$ such that $\varphi(\mathfrak{u})+c\leq -\delta$. Fix any such $\mathfrak{u}$ and define the control $u^*(\cdot) \equiv \mathfrak{u}$. In addition, set $\tau^* \equiv \infty$ (e.g., $\tau^* = \tau_0$ of \eqref{def_hitting_time} for $\pi \in (0, 1)$, which is a.e. infinite by Feller's test for explosions \cite[Theorem 5.5.29]{BMSC}). Since $g(\cdot)\geq0$ is bounded, we have
    \begin{equation}
        J\big(\pi,\Pi^{u^*}_\pi(\cdot), \tau^*\big) = \mathbb{E}^u\left[g(\Pi_{\pi}^{u^*}(\tau^*))+(\varphi(\mathfrak{u}) + c) \, \tau^*\right]\leq \mathbb{E}^u\left[g(\Pi_{\pi}^{u^*}(\tau^*))-\delta\tau^*\right] =-\infty.
    \end{equation}
    Thus, $V(\cdot)\equiv -\infty$ and the value $-\infty$ is always attainable by a suitable constant control and $\tau^*$.
\end{proof}

\begin{remark}\label{rem_zero_control}
    A close inspection of the proof of Theorem~\ref{thm_main} shows that the exclusion $0\notin\mathcal U$ is not essential if we define $M:=\inf_{u\in\mathcal{U}\setminus\{0\}}\eta(u)$; it merely avoids additional casework and technicalities. Intuitively, when $\inf_{u\in\mathcal{U}}\{\varphi(u)+c\}>0$ the choice $u(\cdot)\equiv 0$ on any set of positive $(dt\otimes\mathbb P^u)$-measure is suboptimal: the posterior does not update while positive running cost accrues. Making this precise requires progressively measurable modifications of controls and attention to positive-measure sets with empty interior (e.g., fat Cantor sets). On the other hand, pathologies arise if $\varphi(u)+c\le 0$ for some $u$. For example, if $\varphi(u)+c>0$ for all $u\neq 0$ but $\varphi(0)+c<0$, then $M\geq 0$ yet the choice $u(\cdot)\equiv 0$ together with $\tau\equiv\infty$ is optimal and attains $V\equiv -\infty$. These peculiarities can be ruled out on a case-by-case basis, but for clarity of exposition we refrain from pursuing these technical extensions here.
\end{remark}

\begin{remark}
    Suppose that the infimum of $\eta(\cdot)$ over $\mathcal{U}$ is positive and attained by $u^*\in\mathcal{U}$ for $u^*>0$ (an analogous result holds for $u^*<0$), so that the optimal stopping time is given by $\tau_{A^*} = \inf\left\{t \ge 0 : \Pi^{u^*}_\pi(t) \notin (A^*, 1 - A^*)\right\}$. Let $\mathbb{P}^1$ (resp. $\mathbb{P}^0$) be the (conditional) measure on the original probability space $(\Omega,\mathcal{F},\mathbb{F})$, i.e., of the strong formulation, associated with the case $\theta=1$ (resp. $\theta=0$). If we choose the optimal control $u^*(\cdot)\equiv u^*$ from Theorem \ref{thm_main}, the \textit{likelihood ratio} process $L^{u^*}(\cdot)=(L^{u^*}(t))_{t\geq0}$ for the test satisfies
    \begin{equation*}
        L^{u^*}(t)
        \coloneqq
        \frac{d\mathbb{P}^{1}}{d\mathbb{P}^{0}}
        \bigg|_{\mathcal{F}^{X^{u^*}}(t)}
        =
        \exp\Bigl(u^*\,X^{u^*}(t) - \frac{1}{2}(u^*)^2\,t\Bigr), \quad 0 \le t < \infty,
    \end{equation*}
    and $\Pi^{u^*}(\cdot)$ 
    can be represented as
    $\Pi^{u^*}(t)
    =p\,L^{u^*}(t)/\left(p\,L^{u^*}(t) + (1-p)\right)$  (cf. \cite[Equations (2.6)--(2.8)]{GapShi11}).
    As a result,  we can see that the optimal stopping time $\tau_{A^*}$ is the first time that $X^{u^*}(\cdot)$ hits one of two linear boundaries. In particular, if we define
    \begin{equation}
        \Gamma^*(p)
        \coloneqq
        \ln\left(\frac{(1-A^*)(1-p)}{A^* p}\right) \quad \text{and} \quad \Gamma_*(p)
        \coloneqq
        \ln\left(\frac{A^*(1-p)}{(1-A^*)p}\right)
    \end{equation}
    then by rearranging the representation for $\Pi^{u^*}(\cdot)$ we obtain
    \begin{equation}\tau_{A^*} = \inf\left\{t\ge 0 : X^{u^*}(t) \notin \left(\frac{u^*t}{2}+\frac{\Gamma_*(p)}{u^*}, \frac{u^*t}{2}+\frac{\Gamma^*(p)}{u^*}\right) \right\}. \end{equation}
\end{remark}

\section{Discussion and Examples}\label{sec_examples}

In this section we complement our main results with further analysis and examples. 
We begin by describing the statistical properties of the optimal rule $(u^*,\tau^*,d^*)$, 
focusing on the associated distributions and error probabilities. 
Next, we compare our framework with the $L^2$ loss case of Ekström and Karatzas \cite{EksKar}, 
showing that our results agree whenever the two formulations coincide. 
Finally, we study the quadratic running cost, which permits a parametric analysis of 
the optimal boundaries and value function, supported by numerical illustrations.

\subsection{Statistical properties}

We begin our discussion by addressing the statistical properties of the optimal Bayes rule $(u^*, \tau^*,d^*)$ given in Theorem \ref{thm_main} under the strong formulation. We assume that $\eta(\cdot)$ attains its minimum on $\mathcal{U}$ so that a constant optimal control $u^*\in\mathcal{U}$ exists. We also tacitly assume that the problem has a non-trivial solution with $(A^*,1-A^*)\not=\emptyset$. While this is not necessary, the statements are vacuously true otherwise. Since the optimal control is constant, we restrict our attention to the (conditional) distributions of $\tau^*$ and $d^*$. As in the discussion around \eqref{objective_in_terms_of_pi}, where we wrote $g(\pi) = f(\pi,h(\pi))$, we consider the decision rule $d^*$ in terms of the measurable selector $h(\cdot)$.  We organize all of our results in the next proposition whose proof can be found in Appendix \ref{subsec_proof_of_stat_prop}.

\begin{Prop}\label{prop:stat.prop}
    Fix $A^*< \pi < 1-A^*$. The conditional distribution of the optimal decision rule $d^*\in\{h(A^*),h(1-A^*)\}$ is characterized by the probabilities
    \begin{equation}
        \mathbb{P}\big(d^* = h(1-A^*)\mid\theta=1\big) = \frac{(1-A^*)(\pi-A^*)}{(1-2A^*)\pi}
    \end{equation}
    and
    \begin{equation}
        \mathbb{P}\big(d^* = h(1-A^*)\mid\theta=0\big) =\frac{A^*(\pi-A^*)}{(1-2A^*)(1-\pi)}.
    \end{equation}
    The conditional means of the optimal stopping time $\tau^*$ are given by
    \begin{equation}
        \begin{split}
            \ex\big[\tau^*\mid\theta=1\big]
            &=
            \frac{2}{(u^*)^2}\left[2\ln\left(\frac{1-A^*}{A^*}\right)\frac{(1-A^*)(\pi-A^*)}{(1-2A^*)\pi} - \ln\left(\frac{(1-A^*)\pi}{A^*(1-\pi)}\right)\right], \\
            \ex\big[\tau^*\mid\theta=0\big]
            &=
            \frac{2}{(u^*)^2}\left[\ln\left(\frac{(1-A^*)\pi}{A^*(1-\pi)}\right)-2\ln\left(\frac{1-A^*}{A^*}\right)\frac{A^*(\pi-A^*)}{(1-2A^*)(1-\pi)}\right],
        \end{split}
    \end{equation}
    and the Laplace transforms of the conditional laws are
    \begin{equation}
        \begin{split}
            \ex\big[e^{-\alpha\tau^*}\mid\theta=1\big] 
            &=
            \frac{\left[\left(\frac{1-A^*}{A^*}\right)^{\beta-\frac{1}{2}}-\left(\frac{1-A^*}{A^*}\right)^{-\beta+\frac{1}{2}}\right]\left(\frac{\pi}{1-\pi}\right)^{-\beta-\frac{1}{2}}+\left[\left(\frac{1-A^*}{A^*}\right)^{\beta+\frac{1}{2}}-\left(\frac{1-A^*}{A^*}\right)^{-\beta-\frac{1}{2}}\right]\left(\frac{\pi}{1-\pi}\right)^{\beta-\frac{1}{2}}}{\left(\frac{1-A^*}{A^*}\right)^{2\beta}-\left(\frac{1-A^*}{A^*}\right)^{-2\beta}},
            \\
            \ex\big[e^{-\alpha\tau^*}\mid\theta=0\big] 
            &=
            \frac{\left[\left(\frac{1-A}{A}\right)^{\beta+\frac{1}{2}}-\left(\frac{1-A}{A}\right)^{-\beta-\frac{1}{2}}\right]\left(\frac{\pi}{1-\pi}\right)^{-\beta+\frac{1}{2}}+\left[\left(\frac{1-A}{A}\right)^{\beta-\frac{1}{2}}-\left(\frac{1-A}{A}\right)^{-\beta+\frac{1}{2}}\right]\left(\frac{\pi}{1-\pi}\right)^{\beta+\frac{1}{2}}}{\left(\frac{1-A^*}{A^*}\right)^{2\beta}-\left(\frac{1-A^*}{A^*}\right)^{-2\beta}},
        \end{split}
    \end{equation}
    where $\beta := \sqrt{\frac{2\alpha}{(u^*)^2}+\frac{1}{4}}$. As a consequence, the conditional densities admit the series representations\footnote{Note that the constant $\boldsymbol{\pi}=3.14159...$ should not be confused with the parameter $\pi$ in this expression.}
    \begin{align*}
        \pr\big(\tau^*\in dt \mid \theta=1\big)
        &=
        e^{-\frac{(u^*)^2 t}{8}}\Bigg[\sqrt{\frac{A (1-\pi)}{(1-A)\pi}}\sum_{k=-\infty}^\infty \frac{\ln\left(\frac{\pi}{1-\pi}\left(\frac{1-A^*}{A^*}\right)^{4k+1}\right)}{\sqrt{2\boldsymbol{\pi}} |u^*| t^{3/2}}e^{-\frac{\ln\left(\frac{\pi}{1-\pi}\left(\frac{1-A^*}{A^*}\right)^{4k+1}\right)^2}{2(u^*)^2t}}\\
        & \quad + \sqrt{\frac{(1-A)(1-\pi)}{A\pi}}\sum_{k=-\infty}^\infty \frac{\ln\left(\frac{1-\pi}{\pi}\left(\frac{1-A^*}{A^*}\right)^{4k+1}\right)}{\sqrt{2\boldsymbol{\pi}} |u^*| t^{3/2}}e^{-\frac{\ln\left(\frac{1-\pi}{\pi}\left(\frac{1-A^*}{A^*}\right)^{4k+1}\right)^2}{2(u^*)^2t}}\Bigg]dt,
    \end{align*}
    and
    \begin{align*}
        \pr\big(\tau^*\in dt \mid \theta=0\big)
        &=
        e^{-\frac{(u^*)^2t}{8}}\Bigg[\sqrt{\frac{(1-A)\pi}{A (1-\pi)}}\sum_{k=-\infty}^\infty \frac{\ln\left(\frac{\pi}{1-\pi}\left(\frac{1-A^*}{A^*}\right)^{4k+1}\right)}{\sqrt{2\boldsymbol{\pi}} |u^*| t^{3/2}}e^{-\frac{\ln\left(\frac{\pi}{1-\pi}\left(\frac{1-A^*}{A^*}\right)^{4k+1}\right)^2}{2(u^*)^2t}}\\
        & \quad +\sqrt{\frac{A\pi}{(1-A)(1-\pi)}}\sum_{k=-\infty}^\infty \frac{\ln\left(\frac{1-\pi}{\pi}\left(\frac{1-A^*}{A^*}\right)^{4k+1}\right)}{\sqrt{2\boldsymbol{\pi}} |u^*| t^{3/2}}e^{-\frac{\ln\left(\frac{1-\pi}{\pi}\left(\frac{1-A^*}{A^*}\right)^{4k+1}\right)^2}{2(u^*)^2t}}\Bigg]dt.
    \end{align*}
\end{Prop}

\begin{remark}
Recall that in the classic sequential testing framework (i.e., $\mathfrak{L}(\theta,d) = \mathbf{1}_{d\not=\theta}$) we have that $h(1-A^*) = 1$ and $h(A^*) = 0$. In other words, if the posterior probability process hits $A^*$, the optimal decision $d^* = 0$, and vice versa if the posterior probability process hits $1- A^*$, the optimal decision $d^* = 1$. Therefore, in that case, Proposition \ref{prop:stat.prop} says that for the null hypothesis $H_0: \theta = 0$,
\begin{equation}\pr(\mathrm{Type \ I \ Error})= \pr(\mathrm{Reject} \ H_0 \mid \theta=0) = \frac{A^*(\pi-A^*)}{(1-2A^*)(1-\pi)}\end{equation} and  \begin{equation}\pr(\mathrm{Type \ II \ Error})= \pr(\mathrm{Fail \ to \ Reject} \ H_0 \mid \theta=1) = 1-\frac{(1-A^*)(\pi-A^*)}{(1-2A^*)\pi}.\end{equation}
In particular, the power of the test is $(1-A^*)(\pi-A^*)/\left[(1-2A^*)\pi\right]$.
\end{remark}

\begin{remark}
    Proposition \ref{prop:stat.prop} gives the tools necessary to study the consequences of mispecifying the prior $\pi$ (note $A^*$ and $u^*$ are independent of $\pi$) and the robustness of the solution to the problem inputs $c>0$, $\varphi(\cdot)$ and $\mathfrak{L}(\cdot,\cdot)$ (which determine $d^*$, $A^*$ and $u^*$). To this end, we emphasize that the unit drift and diffusion in the dynamics of $X^u(\cdot)$ of \eqref{controlled_sde} are chosen without loss of generality, so the sensitivity to such assumptions can still be assessed using our results. 
    Indeed, in view of Remark \ref{rem:gen.drift.vol.filtering.eqn}, for drift $\mu\neq 0$ and diffusion coefficient $\sigma>0$, we define the rescaled control set 
    $\tilde{\mathcal U}\coloneqq (\mu/\sigma)\,\mathcal U$, and the transformed running cost $\tilde\varphi(\tilde u)\;\coloneqq\;\varphi\!\left(\sigma \tilde u / 
    \mu\right)$.
    Then, our problem with data $(\mu,\sigma,\varphi(\cdot),\mathcal U)$ is equivalent to the
    unit-coefficients problem $(1, 1, \tilde\varphi(\cdot), \tilde{\mathcal U})$:
    by rescaling the auxiliary control $\tilde u(\cdot) \coloneqq (\mu/\sigma)\,u(\cdot)$,
    the state equations for the posterior (as in \eqref{cond_prob_dynamics} and \eqref{cond_prob_dynamics_scaled}) are identical in $\tilde u$, and the objective values similarly agree.
    Consequently, the value function, optimal stopping boundary $A^*$, and optimal control profile are
    unchanged after the substitution $u^*(\cdot) = (\sigma/\mu)\,\tilde u^*(\cdot)$. Sensitivity with respect to the signal--to--noise ratio $\mu/\sigma$ is therefore captured entirely by the corresponding reparameterization of $\varphi(\cdot)$ and of $\mathcal U$.
\end{remark}

\subsection{The $L^2$ penalty with constrained control}

We take the opportunity to compare our results with those of Ekström and Karatzas \cite{EksKar}. In \cite{EksKar} the authors treat a similar estimation problem for a \textit{general} random variable $\theta$ in the specific case of an $L^2$ classification loss $\mathfrak{L}(\theta,d) = (\theta - d)^2$. Strictly speaking, they consider the penalty (in our notation) $g(\pi)  = (\theta - \pi)^2$ directly, but it appears that these formulations are equivalent since the optimal decision rule for $L^2$ losses is always the posterior expectation (cf. \cite[Theorem 4.1.15]{durrett2019probability}). It turns out that the $L^2$ penalty lends enough structure to the problem to permit a complete analysis for general prior distributions. On the other hand, we are able to retain full tractability for very general losses, but at the cost of restricting to binary classification problems. 

In \cite{EksKar} they study controls $u(\cdot)$ constrained to lie in $(0,1]$ and find that the full bang control $u^*(\cdot)\equiv 1$ is optimal under a growth assumption on the full running cost $\zeta(u) := \varphi(u) + c$. We show here that in the settings where our work and \cite{EksKar} overlap, i.e., for $\mathcal{U}=(0,1]$ and $\theta\sim\mathrm{Bernoulli}(\pi)$, our results agree. 

The running cost of \cite{EksKar} is chosen to satisfy:
\begin{enumerate}
    \item[(i)] $\zeta:(0,1]\to (0,\infty)$ is continuous and non-decreasing, and
    \item[(ii)] there exists a $u_0\in(0,1]$ such that $\frac{\zeta(u)}{u^2}\geq \frac{\zeta(u_0)}{u_0^2}$ for all $u\in(0,1]$.
\end{enumerate}
When (ii) holds with $u_0=1$ the authors call it a \textit{superquadratic cost of control} since it holds that $\zeta(u)\geq \zeta(1) u^2$ for all $u\in(0,1]$. In our notation, we see that (ii) implies that 
\begin{equation}\eta(u) = \frac{\zeta(u)}{u^2} \geq \frac{\zeta(u_0)}{u^2_0}=\eta(u_0).\end{equation}
Namely, the minimum of $\eta(u)$ on $\mathcal{U}=(0,1]$ is attained at $u_0$. In agreement with our own Theorem \ref{thm_main},  Theorem 5.2 of \cite{EksKar} concludes that $u^*(\cdot) \equiv u_0$ is optimal. In particular, when $u_0=1$ we also recover the optimality of the \textit{full bang control} $u^*(\cdot)\equiv 1$. The optimality of the first hitting time of the posterior expectation to two constant boundaries for the $L^2$ loss and Bernoulli priors is also emphasized in their discussion in Section 6.2 of \cite{EksKar}.

\subsection{Quadratic running costs}

We close with an example where the cost of control is given by a quadratic function
\begin{equation}
    \varphi(u) = au^2 + bu,
\end{equation}
with $a > 0$, $b \in \mathbb{R}$, and a fixed time cost is given by $c>0$.  Despite its simplicity, this cost structure captures several essential features commonly observed in information acquisition problems. The parameter $c$ represents the fixed cost of continuing the test, $b$ corresponds to the baseline marginal cost per unit of effort, and $a>0$ accounts for the accelerating cost associated with improving precision when already exerting a high level of control.

\begin{figure}[!ht]
  \centering
  \includegraphics[scale=0.14]{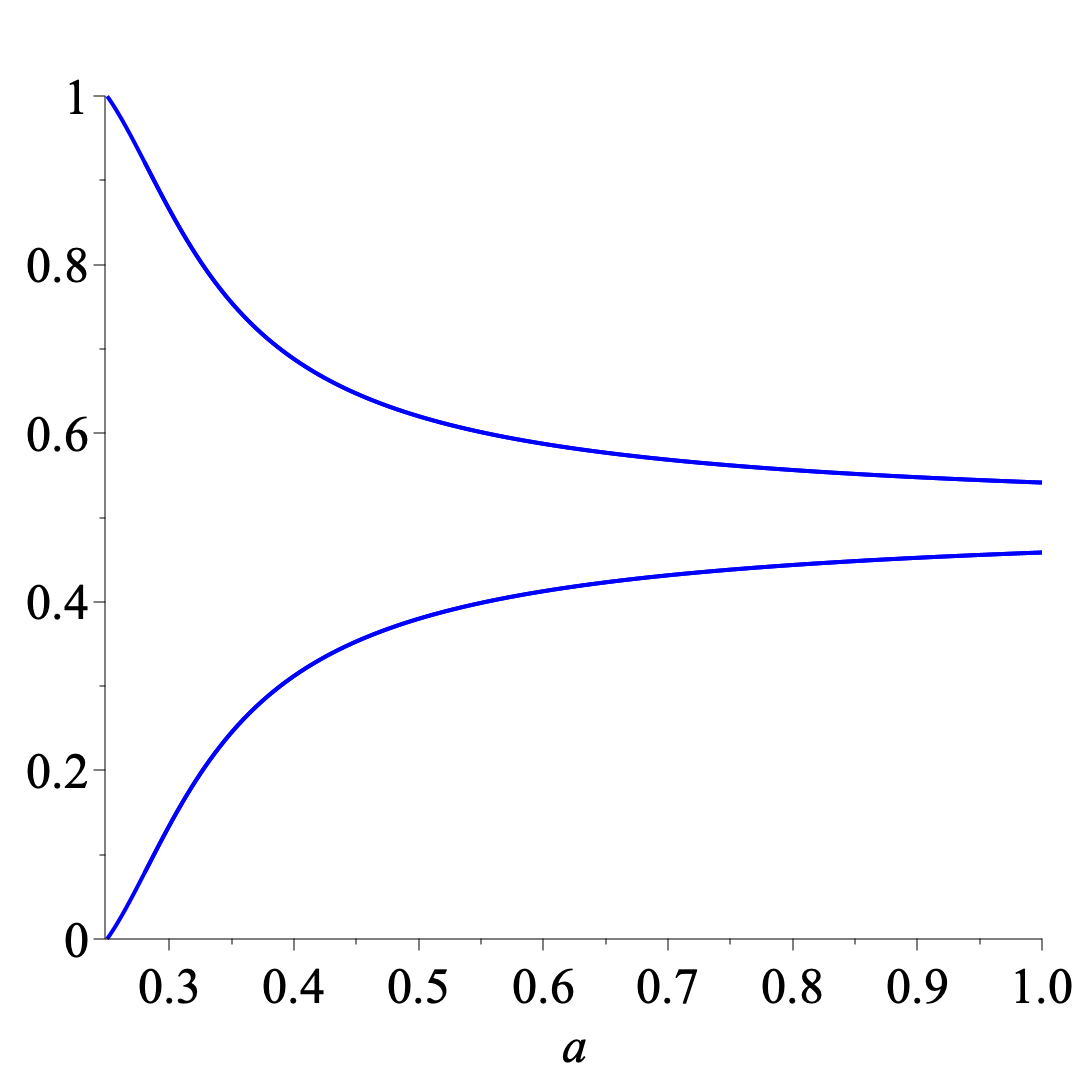}
  \includegraphics[scale=0.14]{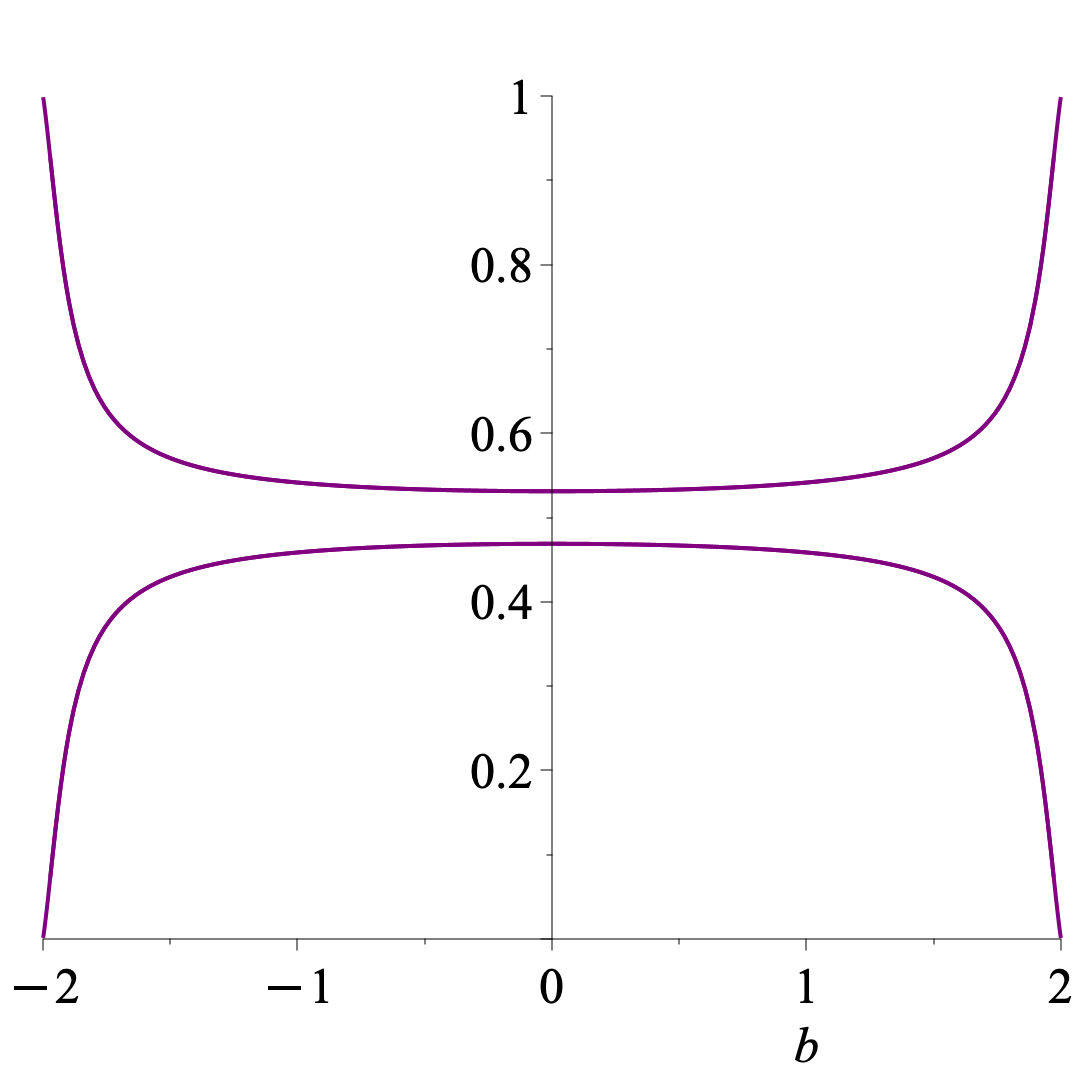}
  \includegraphics[scale=0.14]{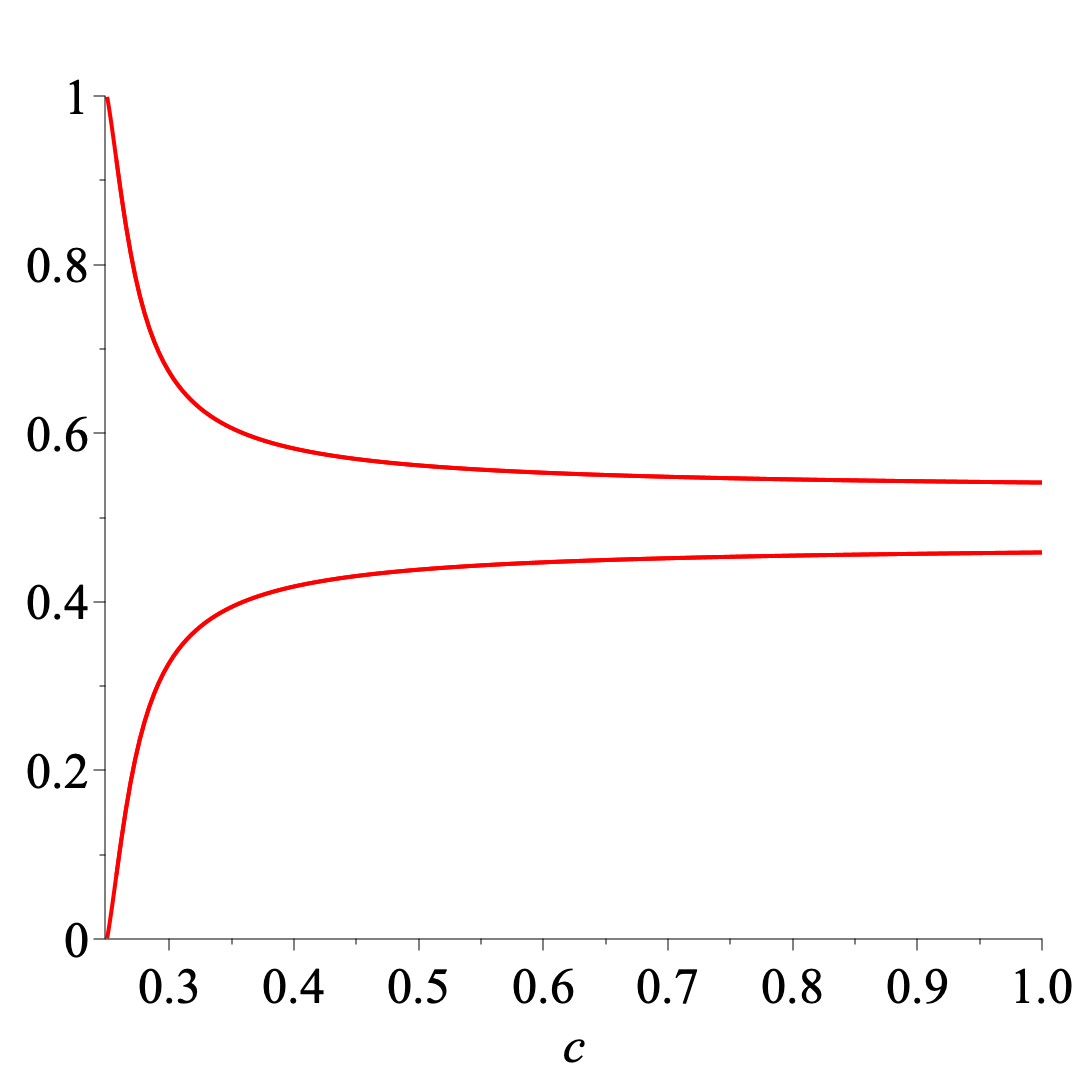}
  \caption{Optimal boundaries $(A^*, 1-A^*)$ as functions of the model parameters when $g(\pi) = \pi\wedge(1-\pi)$ and $b^2-4ac<0$. (Left) Dependence on $a$ when $b=c=1$. (Middle) Dependence on $b$ when $a=c=1$. (Right) Dependence on $c$ when $a=b=1$.}
  \label{fig:opt.bdys.quad}
\end{figure}

\begin{figure}[!ht]
  \centering
  \includegraphics[scale=0.14]{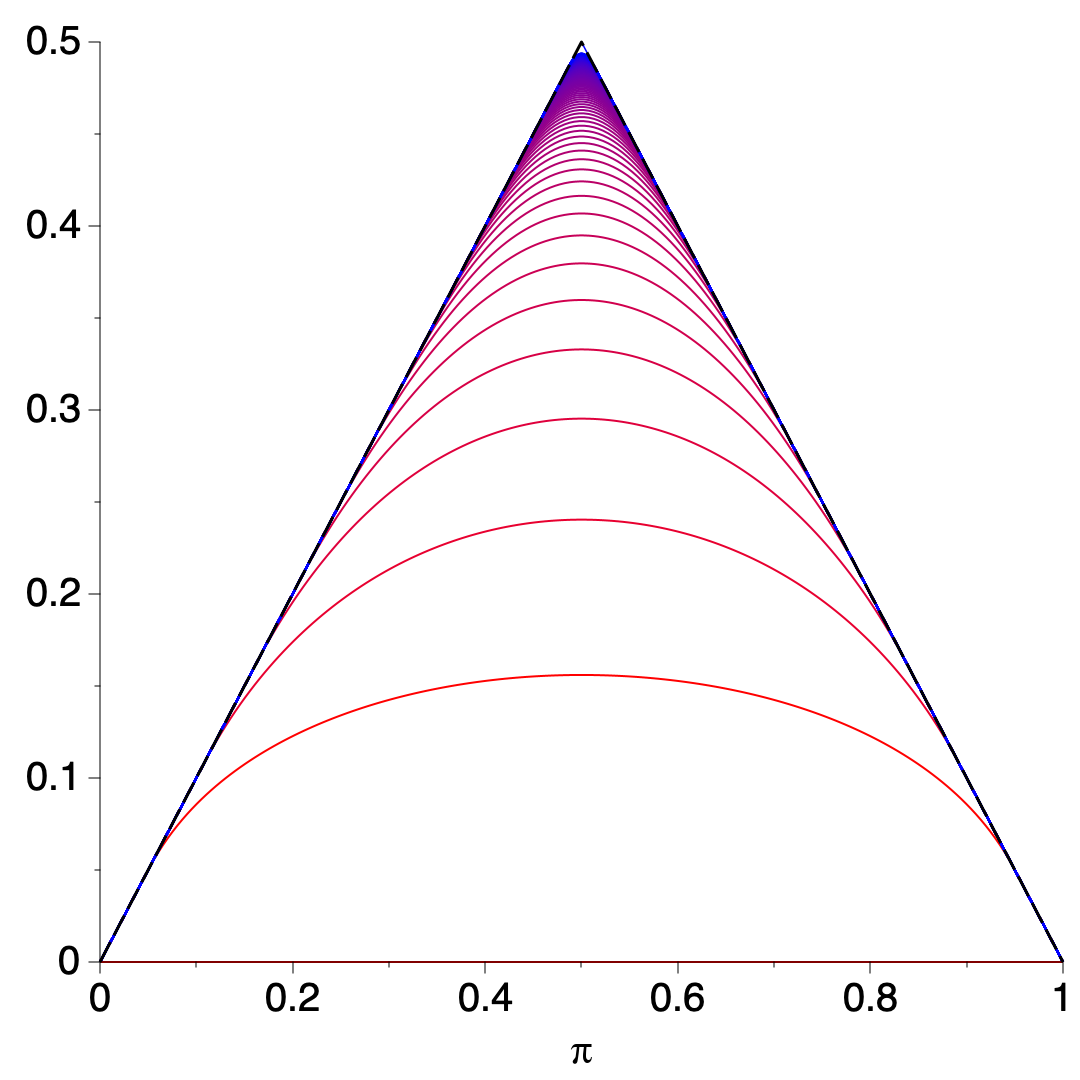}
  \hspace{5em}
  \includegraphics[scale=0.14]{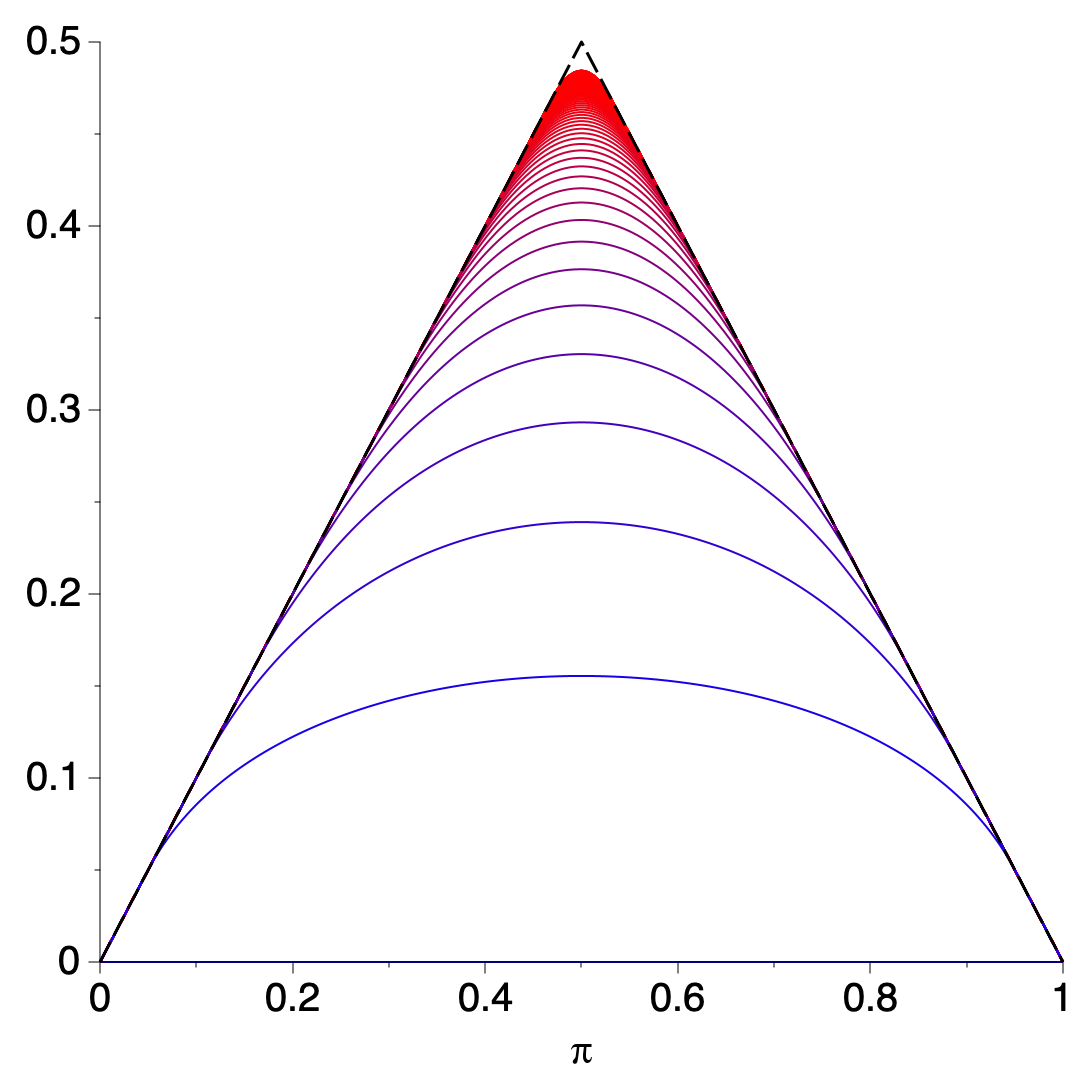}
  \caption{Value function $V(\pi)$ for  $a\in[1/4,5]$ given $b=c=1$ on a color gradient from red to blue (left), and for $b\in[0,2]$ given $a=c=1$ on a color gradient from red to blue (right).}
  \label{fig:opt.value.quad}
\end{figure}

For simplicity, we choose $\mathcal{U}=\mathbb{R}\setminus\{0\}$ to correspond to the unconstrained case. In this setting, the function $\eta(\cdot)$ of \eqref{definition_eta_function} is minimized at $u^* = -2c/b$ with the minimum $M = -(b^2 - 4ac)/4c$. Notably, the optimal control does \textit{not} depend on $a$ and the value of $M$ depends on the discriminant of the quadratic $\varphi(u) + c$, leading to distinct regimes depending on the sign of $b^2 - 4ac$. When the discriminant is negative ($b^2 - 4ac < 0$), a non-trivial solution to the classification problem may exist depending on the cost function $g(\cdot)$. For instance, in the hard classification setting, when $g(\pi) = \pi \wedge (1-\pi)$, the continuation region is always non-empty. On the other hand, when $b^2 -4ac\geq0 $ the problem always degenerates.

To illustrate the solution, we assume that we are faced with the classical terminal penalty $g(\pi) = \pi \wedge (1-\pi)$.  Figure \ref{fig:opt.bdys.quad} displays the optimal stopping boundaries $A^* = A^*(a, b, c)$ ($1-A^* = 1-A^*(a, b, c)$, resp.) as functions of the input parameters. We recall here that $A^*$ can be obtained from Theorem \ref{thm_main} as the smallest solution on $(0, 1/2]$ of the equation
\begin{equation}\label{quadratic_case_boundary_explicit}
    1 = -\frac{b^2 - 4ac}{4c} \left( 
        \frac{1-2\pi}{\pi(1-\pi)} - 2\ln\left(\frac{\pi}{1-\pi}\right)
    \right).
\end{equation}

Figure~\ref{fig:opt.value.quad} shows the corresponding value function $V(\pi)$ for different parameter choices. We highlight here several limiting cases to illustrate how the solution responds to parameter variation.

\begin{itemize}
    \item As $a \to \infty$, control becomes prohibitively expensive. The optimal stopping times converge to $0$ and the value function approaches the terminal penalty: $V(\pi) \to g(\pi)$.
    
    \item As $b \to 0$ or $c\uparrow\infty$, the optimal  control becomes increasingly aggressive, $u^* \to \pm \infty$ (depending on the sign of $b$) and $M \to a$. Despite the divergence of the controls, as long as $a>0$ the optimal stopping boundaries asymptote to a value that is bounded away from $1/2$ (more precisely, to a solution of \eqref{quadratic_case_boundary_explicit} with $b = 0$). On the region enclosed by these limiting boundaries the value function converges to a limit that lies strictly below $g$.

    \item As $b\to \pm 2\sqrt{ac}$, the boundaries diverge to $0$ and $1$, and the problem value $V(\pi)$ converges to $0$.
\end{itemize}

\begin{figure}
    \centering
  \includegraphics[scale=0.14]{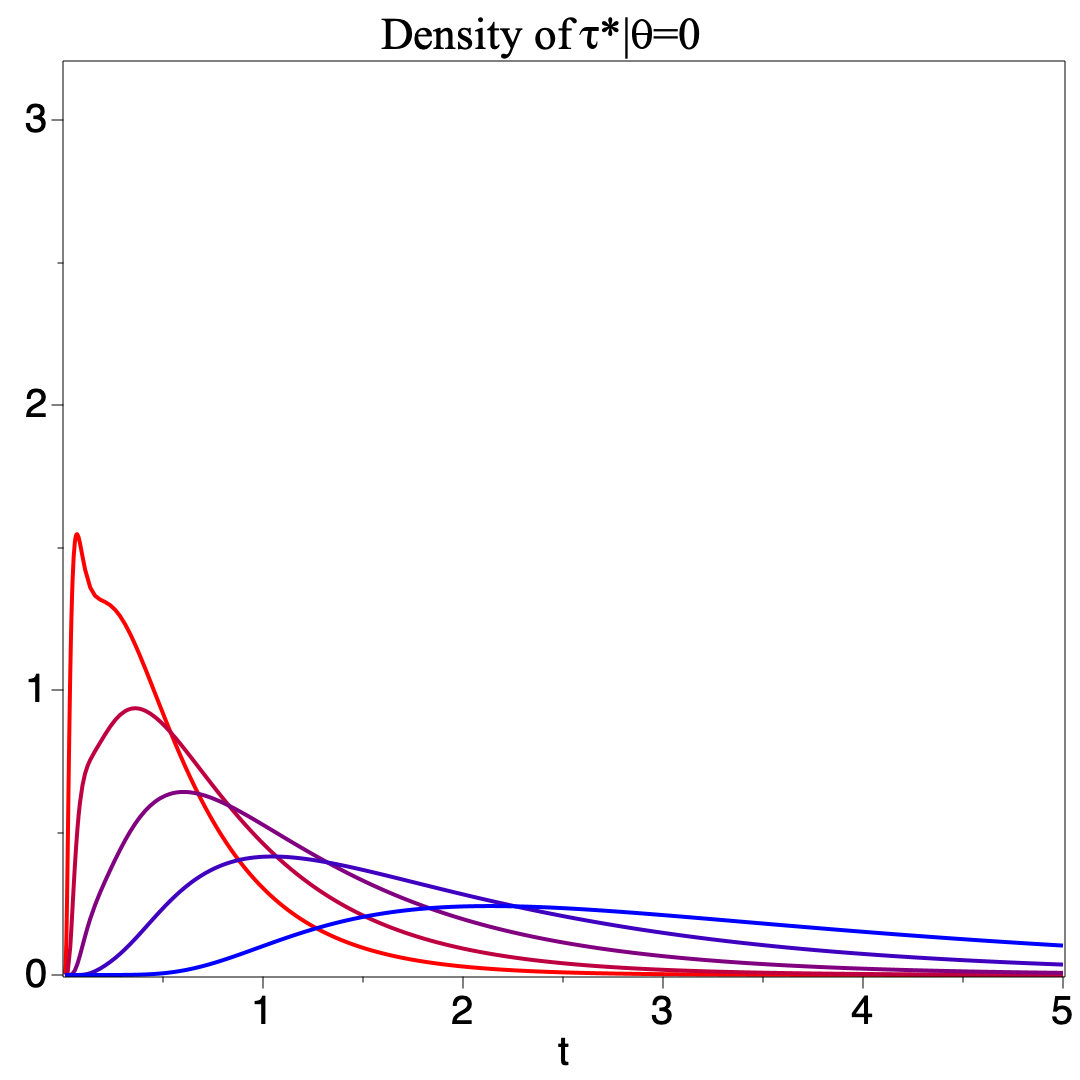}
  \includegraphics[scale=0.14]{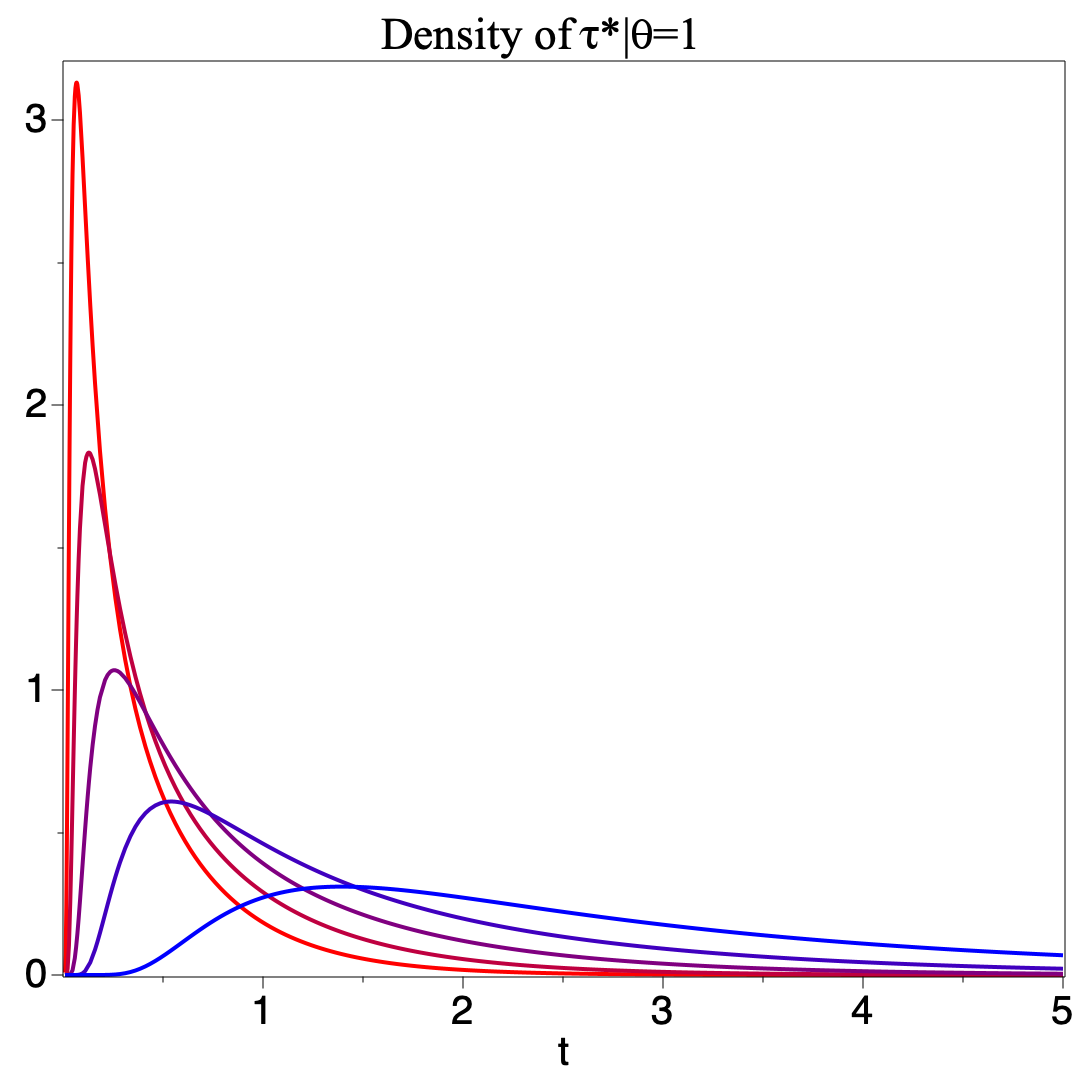}
  \includegraphics[scale=0.14]{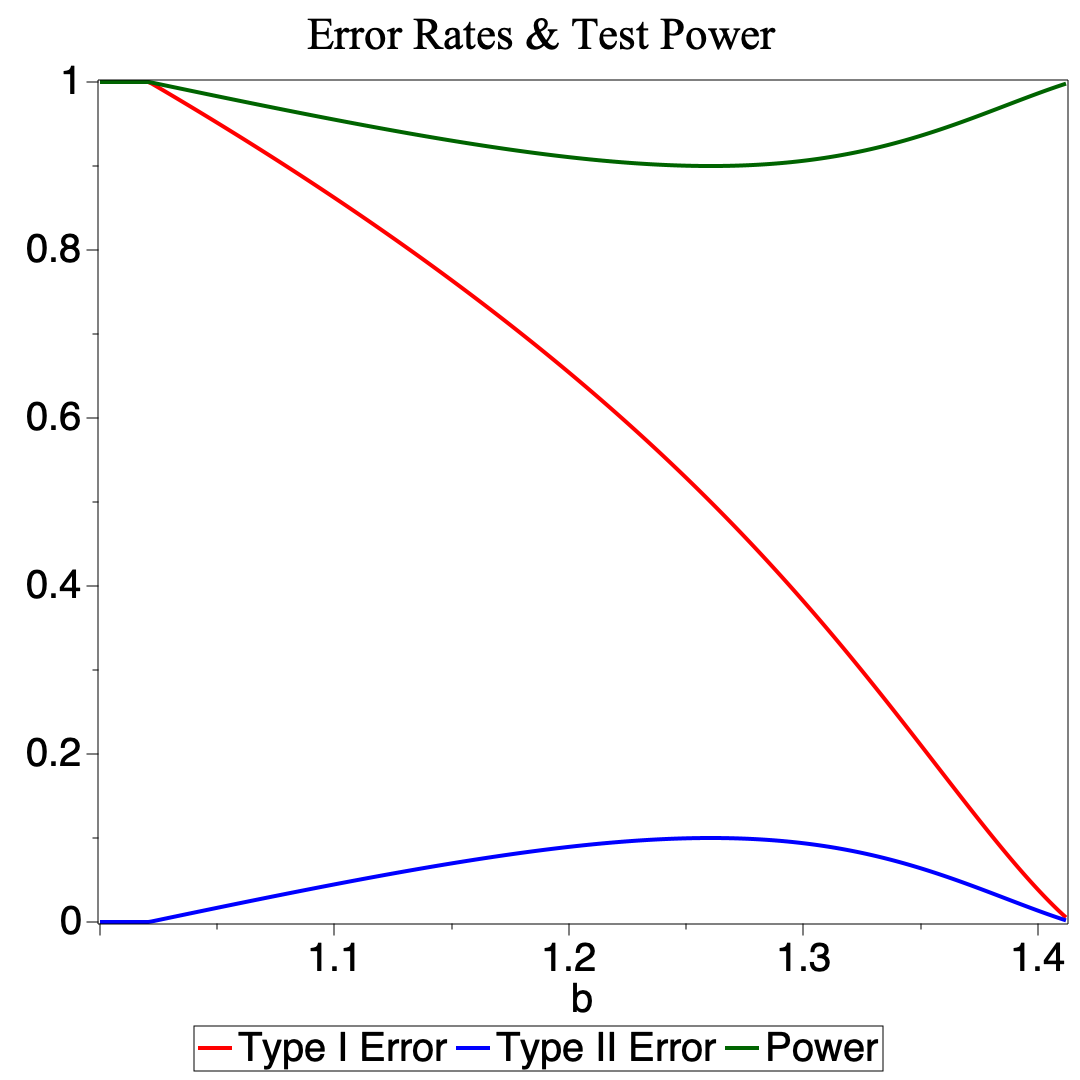}
  \caption{Density of $\tau^* \mid \theta=0$ (left) and $\tau^* \mid \theta=1$ (middle) for the classic loss when $\pi=0.625$ illustrated on a color gradient from red to blue for various values of $b$. Here $a=1/2$, $c=1$ and the parameter $b$ is chosen in an increasing fashion from the set $\{1.28,1.31,1.34,1.37,1.4\}$. (Right) Corresponding Type I and Type II error probabilities and test power as a function of $b\in[1,\sqrt{2}]$. }
  \label{fig:stop.density}
\end{figure}

In our last figure, Figure \ref{fig:stop.density}, we illustrate the evolution of the density of the optimal stopping time $\tau^*$ conditioned on $\theta=1$ and $\theta=0$. We also visualize the behavior of the Type I and Type II errors, and the test power. In these plots, the prior probability $\pi=0.625$ is biased towards the conclusion $\theta=1$. Since the posterior process is started closer to the boundary $1-A^*$, these results in an asymmetry in the shape of the conditional distributions. We recall that as $|b|$ increases, the magnitude of the control, $|u^*|=2c/|b|$, decreases. At the same time, from Figure \ref{fig:opt.bdys.quad}, the boundaries also diverge towards $0$ and $1$. Together, these effects cause the conditional distributions of the optimal stopping time to flatten and shift to the right as $b\to 2\sqrt{ac}$. 

In the rightmost panel of Figure \ref{fig:stop.density} we see the corresponding effect on the test errors. When $b$ is too small, $\pi$ falls into the stopping region and the test is never conducted. This leads to a Type I error of 1 and a Type II error of 0 (equivalently, a power of 1). We see that Type I error decreases monotonically to $0$ as $b\to2\sqrt{ac}$, while Type II error first increases before decreasing again in the limit. This reflects that in the limit $b\to2\sqrt{ac}$ the value of $\theta$ is ultimately determined with certainty.

\section{Conclusion}

We have studied a Bayesian sequential inference problem for the drift of a Brownian motion in which the observer not only chooses their stopping and decision rules, but also controls the rate at which information is acquired. This extension introduces a cost for information acquisition, transforming the classical formulation of Shiryaev \cite{Shiryaev67} into a joint problem of filtering, control, and optimal stopping. Problems of this kind arise in adaptive experimental design and real-time inference across domains such as healthcare, finance, and engineering, where information is costly and decision-making must be timely. 

Our framework accommodates a broad class of classification losses and general running costs on the control process. Interestingly, despite the added complexity introduced by controlling the flow of information, we arrive at a semi-explicit solution that retains the structural elegance of the classical setting. This work thus provides a rare example of a fully tractable triple problem, combining inference, control, and stopping, and invites further exploration of decision-making under information and resource constraints.

We conclude by pointing to three directions for further work, which we find interesting. 
First, the structural resemblance between our results and those of Ekström and Karatzas \cite{EksKar} raises the question of how far this parallel extends. In particular, it would be valuable to clarify how generality in the prior or cost structure impose constraints on the other.
Second, a game version of our model, where the controller and stopper are adversaries, seems both appealing and potentially amenable to analysis. 
Finally, extending our triple problem to a finite or random horizon setting would likely change the character of the problem in a fundamental way, calling for different methods, yet offer promising ground for further exploration.

\vspace{0.5cm}
\noindent
\textbf{Acknowledgments} We are most grateful to Ioannis Karatzas for his very careful readings and many valuable suggestions. S.~Campbell also acknowledges support from an NSERC Postdoctoral Fellowship (PDF‑599675-2025) and a Columbia University CFDT Research Grant.

\begin{appendices}

\section{Proof of Lemma \ref{lem:pi.consistent}}\label{subsec_proof_of_consistency_pi}

\begin{proof}
    Fix a signal intensity $u(\cdot)$ and using the process $X^u(\cdot)$ of \eqref{controlled_sde} define the \textit{likelihood ratio} process
    \begin{equation}
        L^u(t) \coloneqq \exp\left\{\int_0^t u(s) \, dX^u(s) -\frac{1}{2}\int_0^tu(s)^2 \, ds\right\}, \quad 0 \le t <\infty.
    \end{equation}
    Usually, in filtering theory this process appears as the likelihood ratio of the “signal-on” ($\theta=1$) hypothesis relative to the “signal-off” ($\theta=0$) alternative (cf.\ \cite{EksKar, GapShi11}). Namely, let $\mathbb{P}^1$ and $\mathbb{P}^0$ be the (conditional) measures on the original probability space $(\Omega,\mathcal{F},\mathbb{F})$ associated to $\theta=1$ and $\theta=0$, respectively, so that $\mathbb{P}=\pi\mathbb{P}^1+(1-\pi)\mathbb{P}^0$. If $L^u(\cdot)$ were a true $\mathbb{P}^0$--martingale with $\mathbb{E}^{\mathbb{P}^0}[L^u(t)]=1$ (e.g., under Novikov’s or Kazamaki’s criterion), then it is the density process of the equivalent measure $\mathbb{P}^1$ on $(\Omega,\mathcal{F}^{X^{u}}(t))$ :
    \[
      \frac{d\mathbb{P}^1}{d\mathbb{P}^0}\Bigg|_{\mathcal{F}^{X^{u}}(t)} = L^u(t), \quad 0 \le t < \infty.
    \]
    Since we prefer to keep minimal assumptions on $u(\cdot)$, we will not impose these additional martingale conditions here and simply regard $L^u(\cdot)$ as the stochastic exponential above.
    
    We introduce the auxiliary process
    \begin{equation}
        Z^u(t) \coloneqq \int_0^t u(s)\, dX^u(s), \quad 0 \le t <\infty,
    \end{equation}
    so that we can write
    \begin{equation}
        L^u(t) = \exp\left\{Z^u(t) - \frac{1}{2}\langle Z^u\rangle(t)\right\}, \ \ \ dL^u(t) = L^u(t)\,dZ^u(t), \quad 0 \le t < \infty.       
    \end{equation}
    Note that $Z^u(\cdot)$ is a semimartingale in the full filtration $\mathbb{F}$ due to the definition of $u(\cdot)$ and \eqref{eqn:sq.int.u}.
    Define the $[0,1]$--valued process
    \begin{equation}
        P^u(t) \coloneqq \frac{p L^u(t)}{p L^u(t) + (1-p)}, \quad 0 \le t < \infty.
    \end{equation}
    By applying It\^o's formula and simplifying, we see that
    \begin{equation}\label{eqn:P}
        dP^u(t) = -P^u(t)^2(1-P^u(t))\,d\langle Z^u\rangle(t) + P^u(t)(1-P^u(t)) \, dZ^u(t), \quad P^u(0) = p.
    \end{equation}
    At the same time, we notice that $\Pi^u(\cdot)$ satisfies the same stochastic differential equation as $P^u(\cdot)$ as can be seen from Lemma \ref{lem_filtering}.
    Since the coefficients of this stochastic differential equation are Lipschitz on $[0,1]$, the solution to \eqref{eqn:P} is strongly unique (cf. \cite[Chapter 5, Theorems 6 and 38]{Pr05}). By this uniqueness, we conclude $\Pi^u(t) \equiv P^u(t) =p L^u(t)/(p L^u(t) + (1-p)), \, 0 \le t < \infty$.
    
    Next, we observe that when $\theta=1$,
    \begin{equation}
        L^u(t) = \exp\left\{\int_0^t u(s) \, dW(s) + \frac{1}{2}\int_0^tu(s)^2 \, ds\right\}, \quad 0 \le t < \infty.
    \end{equation}
    Under our assumptions on $u(\cdot)$, the process $M^u(t) := \int_0^t u(s) \, dW(s), \, 0 \le t < \infty$ is a locally square integrable martingale with $\langle M^u\rangle(t) = \int_0^t u(s)^2 \, ds, \, 0 \le t < \infty$.
    It follows from the strong law of large numbers for local martingales \cite[p. 144, Corollary 1]{LS89} that $\lim_{t\to\infty}L^u(t) = \infty \ a.s.,$ and hence, $\lim_{t\to\infty}\Pi^u(t) =1 \ a.s.$ Similarly, when $\theta = 0$,
    \begin{equation}L^u(t) = \exp\left\{\int_0^t u(s) \, dW(s) -\frac{1}{2}\int_0^tu(s)^2 \, ds\right\}, \quad 0 \le t < \infty,\end{equation}
    so by the same argument $\lim_{t\to\infty}L^u(t) = 0$ and $\lim_{t\to\infty}\Pi^u(t) =0$.
\end{proof}

\section{Proof of Proposition \ref{prop_hitting_time_expect}}\label{subsec_proof_of_hitting_time_prop}

\begin{proof}
    To obtain \eqref{delta_hittinig_expectation}, i.e., to show that 
    \begin{equation}\label{delta_hittinig_expectation_in_prop}
        \ex^u \big[ \tau_\delta \big] = \frac{2}{u^2} \big(\Psi(\pi) - \Psi(\delta)\big)
    \end{equation}
    in the notation of \eqref{def_psi_function}, we use the representation for the expectation of $\tau_\delta$ taken from Chapter 5.5.C of Karatzas and Shreve \cite{BMSC} (see formula (5.59) on p. 344):
    \begin{equation}\label{tau_expectation_via_scale}
        \ex^u \big[ \tau_\delta \big] 
        = 
        -\int_{\delta}^{\pi}(p(\pi)-p(y)) \, m(d y)+\frac{p(\pi)-p(\delta)}{p(1-\delta)-p(\delta)} \int_{\delta}^{1-\delta}(p(1-\delta)-p(y)) \, m(d y).
    \end{equation}
    Here, $p(\cdot)$ and $m(\cdot)$ are the scale function and the speed measure of the process $\Pi^{u^*}_\pi(\cdot)$, respectively, which, in our case, are given by 
    \begin{equation}
        p(\pi) = \pi \quad \text{ and } \quad m(d\pi) = \frac{2 \, d\pi}{(u  \pi (1-\pi))^2}, \quad \pi \in (0, 1).
    \end{equation}
    Plugging this back to \eqref{tau_expectation_via_scale}, we obtain
    \begin{equation}
        \begin{split}
            \ex^u \big[ \tau_\delta \big] 
            &= 
            -\frac{2}{u^2}\int_{\delta}^{\pi}\frac{\pi - y}{(y\, (1-y))^2} \, dy
            +\frac{\pi - \delta}{1 - 2\delta} \frac{2}{u^2}\int_{\delta}^{1-\delta}\frac{1 - \delta - y}{(y\, (1-y))^2} \, dy
            \\&=
            \frac{2}{u^2} \left[ 
                -
                \pi \int_{\delta}^{\pi}\frac{dy}{(y\, (1-y))^2}
                +
                \int_{\delta}^{\pi}\frac{dy}{y(1-y)^2} 
                \right.\\&\quad \quad \left.+
                \frac{(\pi - \delta)(1-\delta)}{1 - 2\delta} \int_{\delta}^{1-\delta}\frac{dy}{(y\, (1-y))^2}
                -
                \frac{\pi - \delta}{1 - 2\delta} \int_{\delta}^{1-\delta}\frac{dy}{y\, (1-y)^2} 
            \right].
        \end{split}
    \end{equation}
    Using the identities
    \begin{equation}
        \int_a^b \frac{dy}{(y\, (1-y))^2} 
        =
        \left[\frac{2y-1}{y(1-y)}-2 \ln \left(\frac{1-y}{y}\right)\right] \Bigg|_a^b
        \quad \text{and} \quad 
        \int_a^b \frac{dy}{y\, (1-y)^2}
        = \left[\frac{1}{1-y} - \ln\left(\frac{1-y}{y}\right)\right] \Bigg|_a^b,
    \end{equation}
    and rearranging terms yields the desired result. 
\end{proof}

\section{Proof of Proposition \ref{prop:stat.prop}}\label{subsec_proof_of_stat_prop}

\begin{proof}
    We leverage the following useful observation. Let 
    \begin{equation}
        Y^u(t) := \mathsf{logit}\left(\Pi^{u}(t)\right) = \ln\left(\frac{\Pi^{u}(t)}{1 - \Pi^{u}(t)} \right), \quad 0 \le t < \infty
    \end{equation}
    be the \textit{log-likelihood ratio} process for an arbitrary constant control $u(\cdot) \equiv u\in\mathbb{R}\setminus\{0\}$ and initial condition $Y^u(0) = y = \mathsf{logit}(p)$. By applying It\^o's formula, and using that $\Pi^u(t) =\mathsf{sigmoid}(Y^u(t))=:S(Y^u(t)), \, 0 \le t < \infty$, we obtain
    \begin{equation}
        dY^u(t) = \frac{u^2}{2}\big(2S(Y^u(t))-1\big)\,dt + u \, dB^u(t), \quad 0 \le t < \infty.
    \end{equation}
    Expanding $B^u(\cdot)$ from Lemma \ref{lem_filtering} we get
    \begin{equation} 
    dY^u(t) = \frac{u^2}{2} (2S(Y^u(t))-1) \, dt + u\left[ \theta u \, dt + dW(t) - uS(Y^u(t)) \, dt\right]=\frac{u^2}{2} (2\theta- 1)\, dt + u \, dW(t).
    \end{equation}
    Then, by using the independence of $\theta$ and $W(\cdot)$, we see that conditioned on $\theta = 1$
    \begin{equation}
        Y^u(t) = y + \frac{1}{2} u^2 t + u W(t), \quad 0 \le t < \infty,
    \end{equation}
    and conditioned on $\theta = 0$
    \begin{equation}
        Y^u(t) = y - \frac{1}{2} u^2 t + u W(t), \quad 0 \le t < \infty.
    \end{equation}
    Given this simple representation, it will be convenient to work with all quantities in terms of $Y^u(\cdot)$. For this, we observe that the first hitting time of $\Pi^u(\cdot)$ to $1-A^*$ (resp.\ $A^*$) is the same as the first hitting time of $Y^u(\cdot)$ to $\Gamma^*:=\mathsf{logit}(1-A^*)\geq0$ (resp.\ $-\Gamma^*$).

    We will treat the case where we condition on $\theta=1$ as the remaining case is similar. The scale function associated with the conditional dynamics of $Y^u(\cdot)$ is $p(x) = 1-\exp(-x)$. By \cite[Equation (5.5.61)]{BMSC}
    \begin{equation}
        \mathbb{P}\big(Y^u(\tau^*) =\Gamma^* \mid \theta =1\big) = \frac{1- \exp(-y-\Gamma^*)}{1 - \exp(-2\Gamma^*)}, \qquad -\Gamma^*\leq y\leq \Gamma^*.
    \end{equation}
    Since $d^*= h(1-A^*)$ if and only if $Y^u(\tau^*) =\Gamma^*$, this gives us the first result of the Proposition up to unwinding the coordinate transformation induced by the logistic function.
    By \cite[Equation (5.5.59)]{BMSC}
    \begin{equation}
        \mathbb{E}\big[\tau^* \mid \theta=1\big] 
        =
        -\int_{-\Gamma^*}^y (p(y) - p(z)) \, m(dz) +\frac{p(y) - p(-\Gamma^*)}{p(\Gamma^*)-p(-\Gamma^*)}\int_{-\Gamma^*}^{\Gamma^*} (p(\Gamma^*) - p(z)) \, m(dz)
    \end{equation}
    where $m(dz) = \frac{2}{u^2}e^z dz$. That is,
    \begin{equation}
        \mathbb{E}\big[\tau^* \mid \theta=1\big]=\frac{2}{u^2}\left[2\Gamma^*\frac{1- \exp(-y-\Gamma^*)}{1-\exp(-2 \Gamma^*)}-(\Gamma^* +y)\right], \ \ \ -\Gamma^*\leq y\leq\Gamma^*.
        \end{equation}
    Finally, we also have the following representation for the Laplace transform (cf. \cite[Page 315]{BS12})
    \begin{equation}
        \mathbb{E}\big[e^{-\alpha\tau^*} \mid \theta=1\big] = \frac{e^{-\frac{1}{2}(\Gamma^*+y)}\sinh\left((\Gamma^*-y)\sqrt{\frac{2\alpha}{u^2}+\frac{1}{4}}\right)+e^{\frac{1}{2}(\Gamma^*-y)}\sinh\left((y+\Gamma^*)\sqrt{\frac{2\alpha}{u^2}+\frac{1}{4}}\right)}{\sinh\left(2\Gamma^*\sqrt{\frac{2\alpha}{u^2}+\frac{1}{4}}\right)}.
    \end{equation}
    By inverting the Laplace transform (see also \cite[Page 315]{BS12}) we get the density,
    \begin{equation}
        \mathbb{P}\big(\tau^*\in dt \mid \theta=1\big)
        =
        u^2e^{-\frac{u^2 t}{8}}\left[e^{-\frac{1}{2}(\Gamma^*+y)}ss_{u^2 t}(\Gamma^*-y,2\Gamma^*)+e^{\frac{1}{2}(\Gamma^*-y)}ss_{u^2t}(y+\Gamma^*,2\Gamma^*)\right] dt,
        \end{equation}
    where
    \begin{equation}
        ss_t(a,b) =\sum_{k=-\infty}^\infty \frac{b-a+2kb}{\sqrt{2\boldsymbol{\pi}}t^{3/2}}e^{-\frac{(b-a+2kb)^2}{2t}},  \quad a<b.
    \end{equation}
    Performing the necessary substitutions to write these expressions in terms of $\pi$ (and repeating the process for the dynamics given $\theta=0$) completes the proof.
\end{proof}

\end{appendices}

\printbibliography

\end{document}